\documentclass{emsprocart}

\usepackage{srcltx}

\contact[christof@math.unam.mx]{Christof Geiss, Instituto de
Matem\'aticas, Universidad Nacional Aut\'onoma de M\'exico,
04510 M\'exico D.F., M\'exico}

\contact[leclerc@math.unicaen.fr]{Bernard Leclerc, Universit\'e de
Caen, LMNO UMR 6139, F-14032 Caen cedex, France}

\contact[schroer@math.uni-bonn.de]{Jan Schr\"oer, Mathematisches Institut,
Universit\"at Bonn, Beringstr. 1, D-53115 Bonn, Germany}

\newtheorem{theorem}{Theorem}[section]
\newtheorem{corollary}[theorem]{Corollary}

\newtheorem{proposition}[theorem]{Proposition}
\newtheorem{conjecture}[theorem]{Conjecture}

\newtheorem{exo}[theorem]{Exercise} 
\newtheorem*{answer}{Answer}        
\newtheorem{example}[theorem]{Example}

\theoremstyle{definition}

\newtheorem{remark}[theorem]{Remark}

\newcommand{\E}{\mathcal{E}}
\newcommand{\Z}{\mathbb{Z}}

\newcommand{\Pro}{\mathbb{P}}
\newcommand{\N}{\mathbb{N}}
\newcommand{\C}{\mathbb{C}}
\newcommand{\RR}{\mathbb{R}}
\newcommand{\F}{\mathbb{F}}
\newcommand{\R}{\mathcal{R}}
\newcommand{\ii}{\mathbf{i}}
\newcommand{\jj}{\mathbf{j}}
\newcommand{\ab}{\mathbf{a}}
\newcommand{\dd}{\mathbf{d}}
\newcommand{\gf}{\mathfrak{g}}
\newcommand{\nn}{\mathfrak{n}}
\newcommand{\la}{\lambda}
\newcommand{\De}{\Delta}

\newcommand{\CC}{\mathcal{C}}
\newcommand{\QQ}{\mathcal{Q}}
\newcommand{\Ss}{\mathcal{S}}
\newcommand{\BB}{\mathcal{B}}

\newcommand{\La}{\Lambda}

\newcommand{\Sub}{\operatorname{Sub}}
\newcommand{\Hom}{\operatorname{Hom}}
\newcommand{\Ext}{\operatorname{Ext}}

\newcommand{\md}{\operatorname{mod}}
\newcommand{\Gr}{\operatorname{Gr}}
\newcommand{\SL}{\operatorname{SL}}

\newcommand{\pr}{\operatorname{pr}}
\newcommand{\add}{\operatorname{add}}
\newcommand{\yy}{\mathbf{y}}

\title[Preprojective algebras and cluster algebras]{Preprojective algebras and cluster algebras}

\author[C. Geiss, B. Leclerc and J. Schr\"oer]{Christof Geiss, Bernard
  Leclerc and Jan Schr\"oer\footnote{We are grateful to A. Skowro\'nski 
and the ICRA committee for encouraging us to prepare this survey for the 
ICRA XII book.}}

\begin{document}

\begin{abstract}
We use the representation theory of preprojective algebras to
construct and study certain cluster algebras related to 
semisimple algebraic groups.
\end{abstract}



\maketitle

\section{Introduction}

Cluster algebras were invented by Fomin and Zelevinsky in 2001 \cite{FZ1}.
One of the main motivations for introducing this new class of
commutative algebras was to provide a combinatorial and algebraic framework
for studying the canonical bases of quantum groups introduced
by Lusztig and Kashiwara \cite{L1,K}
and the notion of total positivity for semisimple algebraic groups
developed by Lusztig \cite{L2}. 

A first attempt to understand cluster algebras 
in terms of the representation theory of quivers was done by
Marsh, Reineke, and Zelevinsky \cite{MRZ}, using a category of 
decorated representations. 
This was quickly followed by the seminal paper of Buan, Marsh, Reiten,
Reineke, and Todorov \cite{BMRRT} who introduced a new family of
triangulated categories attached to hereditary algebras, 
called cluster categories, and  
showed that the combinatorics of cluster mutations arises
in the tilting theory of these cluster categories.
This yielded a categorification of a large family
of cluster algebras: the acyclic cluster algebras.

In this review, we will explain a different but somewhat parallel
development aimed at giving a representation-theoretic treatment
for another class of cluster algebras, namely those discovered by
Berenstein, Fomin and Zelevinsky \cite{BFZ} in relation with their series of
works on total positivity and the geometry of double Bruhat cells
in semisimple groups. 
Instead of the cluster categories we have used the module categories
of the preprojective algebras corresponding to these semisimple
groups, and more generally certain Frobenius subcategories of
these module categories. This allowed us to prove that the cluster
monomials of the cluster algebras we consider belong to the dual
of Lusztig's semicanonical basis, and in particular are linearly
independent \cite{GLSInv}. It also enabled us to introduce new
cluster algebra structures on the coordinate rings of partial
flag varieties \cite{GLSFour}.

\section{Total positivity, canonical bases and cluster algebras}
\label{sect2}
Before reviewing our construction, we would like to illustrate by means
of some simple examples why total positivity, canonical bases and
cluster algebras are intimately connected. 

Lusztig has defined the totally positive part $X_{>0}$
for several classes of complex algebraic varieties $X$ attached to 
a semisimple algebraic group $G$.
The definition uses the theory of canonical bases for irreducible
$G$-modules.
It is not our intention in this survey to explain this definition,
neither to recall the construction of canonical bases
(for excellent reviews of these topics, we refer the reader to
\cite{LI1,LI2}).  
Instead of this, we shall present a few examples
for which both the totally positive varieties and the canonical
bases can be described in an explicit and elementary way. 
We shall then see that a cluster algebra structure on the coordinate
ring naturally arises from this description.  

Our first example is trivial, but nevertheless useful to get started.
Let $V=\C^{2n}$ be an even dimensional vector space
with natural coordinates $(y_1,\ldots, y_{2n})$.
Let $X=\Pro(V)=\Pro^{2n-1}$ be the corresponding projective space.
In this case our group $G$ is $SL(V)$.
The totally positive part of $V$ is simply the orthant
\[
V_{>0} = \{v\in V \mid y_1(v) > 0, \ldots, y_{2n}(v)>0 \},
\] 
and the totally positive part of $X$ is the subset of $X$ consisting of
points having a system of homogeneous coordinates
$(y_1:\cdots:y_{2n})$ with all $y_i$ positive.
The coordinate ring of $V$ (or the homogeneous coordinate
ring of $X$) is $R=\C[y_1,\ldots,y_{2n}]$. 
It has a natural $\C$-basis given by all monomials in the $y_i$'s.
The natural action of $G$ on $V$ makes $R$ into a linear representation
of $G$, which decomposes into irreducible representations as
\[
R=\bigoplus_{k\ge 0} R_k,
\]
where $R_k\cong S^k(V^*)$ is the degree $k$ homogeneous part of $R$. 
For every irreducible representation of $G$, Lusztig and Kashiwara
have introduced a canonical basis and a dual canonical basis
(also called lower global basis and upper global basis by Kashiwara).
It is not difficult to check that, in the case of the simple
representations $R_k$, the dual canonical basis coincides with
the basis of monomials in the $y_i$'s of total degree $k$.

We now pass to a more interesting example.
Consider the nondegenerate quadratic form on $V$ given by
\[
q(y_1,\ldots,y_{2n}) = \sum_{i=1}^n (-1)^{i-1} y_i\,y_{2n+1-i}.
\]
Let 
\[
\CC=\{v\in V \mid q(v) = 0\}
\]
be its isotropic cone and $\QQ = \Pro(\CC)$ the corresponding
smooth quadric in $X=\Pro(V)$. The quadric
$\QQ$~can be seen as a partial flag variety for the special orthogonal
group $H$ attached to the form $q$, and so it has, following Lusztig \cite{L3},
a well-defined totally positive part. 
Let us try to guess what is $\QQ_{>0}$, or equivalently what is $\CC_{>0}$.

It seems natural to require that $\CC_{>0}$ is contained in $V_{>0}$.
But this is not enough, and in general $\CC_{>0}$ is going to be 
a proper subset of $\CC \cap V_{>0}$. 
To see this, we may use the known fact that the totally positive
part of a variety of dimension $k$ is homeomorphic to 
$\RR_{>0}^k$, hence we might expect that it is described
by a system of $k$ inequalities.
However, $\CC \cap V_{>0}$ is the subset of $\CC$ given 
by the $2n$ inequalities $y_i>0$,
and since $\CC$ has dimension $2n-1$, we would like to 
define $\CC_{>0}$ by a system of only $2n-1$ inequalities.
In other words, there should exist a system of $2n-1$ 
functions $(f_1, \ldots, f_{2n-1})$ on $\CC$ such that 
\begin{equation}\label{eq1}
\CC_{>0} = \{v\in \CC \mid f_1(v) > 0, \ldots, f_{2n-1}(v)>0 \}.
\end{equation}
Such a system $(f_1, \ldots, f_{2n-1})$ is called
a \emph{positive coordinate system}.
 
So we are looking for a ``natural'' set of $2n-1$ functions
$(f_1, \ldots, f_{2n-1})$ such that the positivity of 
$f_1(v), \ldots, f_{2n-1}(v)$ implies the positivity of
$y_1(v), \ldots, y_{2n}(v)$.
Let us try this idea in the case $n=3$. 
We have 
\[
q(y_1,\ldots,y_6)= y_1y_6 - y_2y_5 + y_3y_4.
\] 
On $\CC \cap V_{>0}$ we therefore have the relation
\[
y_5 = \frac{y_1y_6 + y_3y_4}{y_2}.
\] 
Hence the positivity of the 5 coordinates in the right-hand side 
implies the positivity of $y_5$, that is, the defining equation of
$\CC$ allows to eliminate the inequality $y_5>0$ from the 6 
defining inequalities of $\CC \cap V_{>0}$. So we could take 
\[
(y_1, y_6, y_3, y_4, y_2) 
\]
as a positive coordinate system.
Note that $y_2$ and $y_5$ play the same role and are exchangeable:
we could also take 
$(y_1, y_6, y_3, y_4, y_5)$. 
This would define the same subset $\CC_{>0}$, which in this case
is simply $\CC \cap V_{>0}$.

Already in the case $n=4$ the same trick no longer works. 
Indeed, the defining equation of $\CC$ is now 
$y_1y_8 - y_2y_7 + y_3y_6 -y_4y_5 =0$, 
which does not allow us to express any of the $y_i$'s as a 
subtraction-free expression in terms of the 7 remaining ones.
To overcome this problem, we introduce a new quadratic
function 
\[
p= y_3y_6 -y_4y_5 = y_2y_7 - y_1y_8
\]
on $\CC$. 
On $\CC \cap V_{>0}$ we then have
\[
y_7 = \frac{y_1y_8 + p}{y_2},\quad y_6 = \frac{y_4y_5 + p}{y_3},
\]
and this leads us to take  
\begin{equation}\label{eq2}
(y_1, y_8, y_4, y_5, p, y_2, y_3) 
\end{equation}
as a positive coordinate system.
Again, $y_2$ and $y_7$ are exchangeable, as are
$y_3$ and~$y_6$. So we would obtain the same
subset $\CC_{>0}$ by using, instead of (\ref{eq2}), 
each of the 3 alternative systems of coordinates
\begin{equation}\label{eq3}
(y_1, y_8, y_4, y_5, p, y_7, y_3),\quad
(y_1, y_8, y_4, y_5, p, y_2, y_6),\quad 
(y_1, y_8, y_4, y_5, p, y_7, y_6).
\end{equation}
Note that in this case our candidate for $\CC_{>0}$ is a proper subset 
of $\CC \cap V_{>0}$, since the positivity of $p$ does not follow from the 
positivity of the $y_i$'s.

It turns out that this naive candidate coincides
with the totally positive part of $\CC$ defined by Lusztig.
To explain this, let us consider the coordinate ring
\[
A = \C[y_1,\ldots,y_8]/(y_1y_8 - y_2y_7 + y_3y_6 -y_4y_5)
\]
of $\CC$, or in other words the homogeneous coordinate ring of
the quadric $\QQ$. As before, $A$ is in a natural way a representation
of the special orthogonal group $H$, and the homogeneous components
$A_k\ (k\ge 0)$ coincide with the irreducible direct summands of
this representation. Hence by putting together the dual canonical bases
of all summands $A_k$, we get a dual canonical basis of $A$.
We claim that in this easy situation, the dual canonical basis
can be explicitly computed and has the following 
simple description. 
Namely,
the dual canonical basis of $A_k$ consists of
all the degree $k$ monomials in 
$y_1, \ldots, y_8, p$ containing only variables of \emph{one} of
the 4 coordinate systems displayed in (\ref{eq2}), (\ref{eq3}). 
Here, $y_1,\ldots, y_8$ have degree 1 and $p$ has degree 2.

For example, the dual canonical
basis of $A_1 \cong V^*$ is 
$\{y_1, y_2, y_3, y_4, y_5, y_6, y_7, y_8\}$,
and the dual canonical basis of $A_2$ consists of $p$ and
of all the degree 2 monomials in the $y_i$'s except $y_2y_7$
and $y_3y_6$. 

Now, Lusztig has shown \cite[Prop.~3.2, Th.~3.4]{L3} that $\CC_{>0}$
has the following characterization: it consists of all
elements $v$ of $\C^{2n}$ such that, for every element $b$ of the dual canonical
basis of $A_k$ and for every $k$, one has $b(v) > 0$.  
Because of the monomial description of the dual canonical basis,
we see that this agrees with our naive definition of
$\CC_{>0}$.

\begin{exo}
{\rm
Guess in a similar way what is the definition of $\CC_{>0}$ for $n\ge 4$.
}
\end{exo}

\begin{answer}
{\rm
For $s=1, 2, \ldots, n-3$, put 
\[
p_s= \sum_{k=1}^{s+1}(-1)^{s+1-k}y_ky_{2n+1-k}.
\]
Then
$\CC_{>0}$ is the subset of $\CC$ given by the following 
$n+1$ inequalities
\[
y_1>0,\quad y_n>0, \quad y_{n+1}>0,\quad y_{2n}>0,\quad
p_s >0,\ (s=1,\ldots, n-3),
\]
together with \emph{one} (it does not matter which one)
of the two inequalities
\[
y_k>0,\qquad y_{2n+1-k}>0,
\]
for each $k=2, 3, \ldots, n-1$.
}
\end{answer}

Thus for every $n$, $\CC_{>0}$ can be described as in (\ref{eq1}) 
by a positive coordinate system, and there are $2^{n-2}$ different but
equivalent such systems. 
In fact, one can also check that the dual canonical basis of
the coordinate ring of $\CC$ consists of all monomials in the
$y_i$'s and $p_s$'s supported on \emph{one} of these $2^{n-2}$
coordinate systems.

The definition of a cluster algebra will be recalled 
in Section~\ref{newsect} below.
A reader already familiar with it will
immediately recognize an obvious cluster algebra
structure on the coordinate ring of $\CC$ emerging from this
discussion.
Its $2(n-2)$ \emph{cluster variables} are 
\[
y_k,\quad y_{2n+1-k},
\qquad (k=2, 3, \ldots, n-1). 
\]
Its \emph{coefficient ring} is generated by 
\[
y_1, y_n, y_{n+1},y_{2n},
p_s,\quad (s=1,\ldots, n-3).
\]
Its \emph{clusters} are the $2^{n-2}$ possible choices of $2n-1$
of these functions forming a positive coordinate system. 
Its \emph{cluster monomials} are 
all the monomials supported on a single cluster, and 
its \emph{exchange relations} are 
\[
y_ky_{2n+1-k} = 
\left\{
\begin{array}{ll}
p_1+y_1y_{2n} & \mbox{if $k=2$,}\\[2mm]
p_{k-1}+p_{k-2} & \mbox{if $3\le k \le n-2$,}\\[2mm]
y_ny_{n+1}+p_{n-3} & \mbox{if $k=n-1$.}
\end{array}
\right.
\]
This is a cluster algebra of finite type (it has finitely many
cluster variables). Its type, according to Fomin and Zelevinsky's
classification \cite{FZ2} is $A_1^{n-2}$.

To summarize, the cone $\CC$ and the corresponding quadric $\QQ$
are examples of algebraic varieties for which Lusztig has described
a natural totally positive subset $\CC_{>0}$ or $\QQ_{>0}$. 
What we have found is that their coordinate ring is endowed 
with the structure of a cluster algebra such that

\begin{quote}
(1) each cluster gives rise to a positive coordinate system;  

(2) the dual canonical basis of the coordinate ring coincides with
the set of cluster monomials.
\end{quote}

This is the prototype of what one would like to do
for each variety $X$ having a totally positive part $X_{>0}$
in Lusztig's sense.
But in general, things become more complicated.
First, the cluster algebra structure, when it is known, 
is usually well-hidden, and its description requires a lot of difficult (but
beautiful) combinatorics. As an example, one may consult the paper 
of Scott \cite{S} and in particular the cluster structures of the Grassmannians
$\Gr(3,6)$, $\Gr(3,7)$ and $\Gr(3,8)$.   
Secondly, these cluster algebras are generally of infinite type
so one cannot hope for a closed and finite description as 
in the above example. This is not too bad if one is mainly interested
in total positivity, since one may not necessarily need to know
\emph{all} positive coordinate systems on $X$. 
But it becomes a challenging issue if one aims at a monomial description of 
the dual canonical basis of the coordinate ring, because that would 
likely involve infinitely many families of monomials.
In fact such a monomial description may not even be possible, since,
as shown in \cite{Le}, there may exist elements of the dual
canonical basis whose square does not belong to the basis.
In any case, even if the cluster structure is known, more work
is certainly needed to obtain from it a full description of the dual canonical
basis. Finally, there is no universal recipe for getting a cluster
structure on the coordinate ring. Actually, the existence
of such a structure is not guaranteed by any general theorem, so
it often seems kind of miraculous when it eventually comes out
of some complicated calculations.

The aim of this review is to explain some recent progress made in these
directions by means of the representation theory of preprojective
algebras.
We will choose as our main example the partial flag varieties $X$ attached
to a simple algebraic group $G$ of type $A, D, E$. 
Thus $X$ is a homogeneous space $G/P$, where $P$ is 
a parabolic subgroup of $G$. To $G$ one can attach the preprojective
algebra $\La$ with the same Dynkin type. To $P$ (or rather to its 
conjugacy class) one can attach a certain injective $\La$-module $Q$,
and the subcategory $\Sub Q$ of the module category $\md\La$ 
cogenerated by $Q$.
We will show that $\Sub Q$ can be regarded as a categorification
of the multihomogeneous coordinate ring of $X$, and that the rigid
modules in $\Sub Q$ give rise to a cluster structure on this ring.
In particular, this yields a uniform recipe for producing explicit
cluster structures, many of which were first 
discovered in this way. The cluster structure is
of finite type when $\Sub Q$ has finite representation type, and in
these exceptional cases, the Auslander-Reiten quiver of $Q$ is quite
helpful for understanding the ensuing combinatorics. Finally, this approach
allows to show that the cluster monomials of these algebras belong
to the dual of Lusztig's semicanonical basis. Unfortunately, 
the relation between the semicanonical and the canonical basis is a 
subtle question (see \cite{GLSAnnENS}).
Nevertheless, as predicted by the general conjecture
of Fomin and Zelevinsky \cite[p.498]{FZ1}, we believe that the
cluster monomials also belong to the dual canonical basis, that is,
we conjecture that they lie in the intersection of the dual canonical
basis and the dual semicanonical basis (see below \S\ref{sect18}).

\section{Preprojective algebras}

We start with definitions and basic results about preprojective
algebras of Dynkin type.

Let $\Delta$ be a Dynkin diagram of type $A$, $D$ or $E$.
We denote by $I$ the set of vertices and by $n$ its cardinality.
Let $\overline{Q}$ be the quiver obtained from $\Delta$ 
by replacing each edge, between $i$ and $j$ say, 
by a pair of opposite arrows $a : i\longrightarrow j$ and
$a^* : j \longrightarrow i$.
Let $\C\overline{Q}$ denote the path algebra of $\overline{Q}$
over $\C$. We can form the following quadratic element in 
$\C\overline{Q}$,
\[
c= \sum (a^*a - aa^*),
\]
where the sum is over all edges of $\Delta$.
Let $(c)$ be the two-sided ideal generated by~$c$.
Following Gelfand and Ponomarev \cite{GP}, 
we define the \emph{preprojective algebra}
\[
\La := \C\overline{Q}/(c).
\]
It is well-known that $\La$ is a finite-dimensional selfinjective
algebra.
It has finite representation type 
if $\Delta$ has type $A_n\ (n\le 4)$, 
tame type if $\Delta$ has type $A_5$ or $D_4$,
and wild type in all other cases (see \cite{DR}).

We denote by $S_i (i\in I)$ the simple $\La$-modules, and by
$Q_i (i\in I)$ their injective envelopes.
\begin{example}\label{exa31}
{\rm Let $\La$ be of type $D_4$.
We label the 3 external nodes of the Dynkin diagram of type $D_4$
by 1, 2, 4, and the central node by 3. 
With this convention, the socle filtration of $Q_4$ is 
$$
\begin{tabular}{c}
$S_4$\\
\hline
$S_3$\\
\hline
$S_1 \oplus S_2$\\
\hline
$S_3$\\
\hline
$S_4$
\end{tabular}
$$
and the socle filtration of $Q_3$ is
$$
\begin{tabular}{c}
$S_3$\\
\hline
$S_1 \oplus S_2 \oplus S_4$\\
\hline
$S_3 \oplus S_3$\\
\hline
$S_1 \oplus S_2 \oplus S_4$\\
\hline
$S_3$
\end{tabular}
$$
The structure of $Q_1$ and $Q_2$ can be obtained from
that of $Q_4$ by applying the order 3 diagram automorphism
$1 \mapsto 2 \mapsto 4 \mapsto 1$. 
}
\end{example}

One important property of the module category $\md\La$ is
the following $\Ext$-symmetry.
Let $D$ denote duality with respect to the field $\C$.
We have 
\begin{equation}
\Ext^1_\La(M,N) \cong D\Ext^1_\La(N,M), \qquad (M,N\in \md\La),
\end{equation}
and this isomorphism is functorial with respect to $M$ and $N$
(see \cite{GLSComp}).

\section{Regular functions on maximal unipotent subgroups}

We turn now to semisimple algebraic groups.
For unexplained terminology, the reader can consult standard
references, \emph{e.g.} \cite{CSM}, \cite{FH}, \cite{Hu}.

Let $G$ be a simply connected simple complex algebraic group
with the same Dynkin diagram $\Delta$ as $\La$. 
Let $N$ be a fixed maximal unipotent subgroup of $G$.
If $G=\SL(n+1,\C)$, we can take $N$ to be the subgroup
of upper unitriangular matrices.
In general $N$ is less easy to describe.
To perform concrete calculations, one can use the one-parameter
subgroups $x_i(t)\ (i\in I,\, t\in \C)$ associated with the simple
roots, which form a distinguished set of generators of $N$.

\begin{example}{\rm
In type $A_n$, if $N$ is the subgroup
of upper unitriangular matrices of $\SL(n+1,\C)$, 
we have $x_i(t) = I + tE_{i,i+1}$, where $I$ is the
identity matrix and $E_{ij}$ the matrix unit with
a unique nonzero entry equal to $1$ in row $i$ and
column $j$.}
\end{example}

\begin{example}\label{exa42}
{\rm
In type $D_n$, $N$ can be identified with the subgroup
of the group 
of upper unitriangular matrices of $\SL(2n,\C)$,
generated by
\[ 
x_i(t) = \left\{ \begin{array}{ll}
I + t(E_{n-i+1,n-i+2}+E_{n+i-1,n+i}) &\text{if } 2\le i\le n,\\
I+t(E_{n-1,n+1}+E_{n,n+2}) &\text{if } i=1.
\end{array}\right.  
\]
Thus in type $D_4$ we can take for $N$ the subgroup of $\SL(8,\C)$
generated by 
\[
x_1(t)=
\begin{pmatrix}
1&0&0&0&0&0&0&0\cr
0&1&0&0&0&0&0&0\cr
0&0&1&0&t&0&0&0\cr
0&0&0&1&0&t&0&0\cr
0&0&0&0&1&0&0&0\cr
0&0&0&0&0&1&0&0\cr
0&0&0&0&0&0&1&0\cr
0&0&0&0&0&0&0&1
\end{pmatrix},\quad
x_2(t)=
\begin{pmatrix}
1&0&0&0&0&0&0&0\cr
0&1&0&0&0&0&0&0\cr
0&0&1&t&0&0&0&0\cr
0&0&0&1&0&0&0&0\cr
0&0&0&0&1&t&0&0\cr
0&0&0&0&0&1&0&0\cr
0&0&0&0&0&0&1&0\cr
0&0&0&0&0&0&0&1
\end{pmatrix},
\]
\[
x_3(t)=
\begin{pmatrix}
1&0&0&0&0&0&0&0\cr
0&1&t&0&0&0&0&0\cr
0&0&1&0&0&0&0&0\cr
0&0&0&1&0&0&0&0\cr
0&0&0&0&1&0&0&0\cr
0&0&0&0&0&1&t&0\cr
0&0&0&0&0&0&1&0\cr
0&0&0&0&0&0&0&1
\end{pmatrix},\quad
x_4(t)=
\begin{pmatrix}
1&t&0&0&0&0&0&0\cr
0&1&0&0&0&0&0&0\cr
0&0&1&0&0&0&0&0\cr
0&0&0&1&0&0&0&0\cr
0&0&0&0&1&0&0&0\cr
0&0&0&0&0&1&0&0\cr
0&0&0&0&0&0&1&t\cr
0&0&0&0&0&0&0&1
\end{pmatrix}.
\]
}
\end{example}

As an algebraic variety, $N$ is isomorphic
to an affine space of complex dimension the number $r$ of positive 
roots of $\Delta$. 
Hence its coordinate ring $\C[N]$ 
is isomorphic to a polynomial ring in $r$ variables.
For example in type $A_n$ if $N$ is the group of unitriangular
matrices, each matrix entry $n_{ij}\ (1\le i<j \le n+1)$ is  
a regular function on $N$ and $\C[N]$ is the ring of polynomials 
in the $n(n+1)/2$ variables $n_{ij}$.

In the general case, the most convenient way of specifying a
regular function $f\in\C[N]$ is to describe its evaluation
$f(x_{i_1}(t_1)\cdots x_{i_k}(t_k))$ at an arbitrary product of 
elements of the one-parameter subgroups.
In fact one can restrict to certain special products.
Namely, let $W$ denote the Weyl group of $G$ and $s_i \ (i\in I)$
its Coxeter generators. Let $w_0$ be the longest element of $W$
and let $w_0 = s_{i_1}\cdots s_{i_r}$ be a reduced decomposition.
Then it is well-known that the image of the map
\[
(t_1,\ldots,t_r)\in\C^r \mapsto x_{i_1}(t_1)\cdots x_{i_r}(t_r)\in N
\]
is a dense subset of $N$. It follows that a polynomial function $f\in\C[N]$
is completely determined by its values on this subset.

\section{A map from  $\md\La$ to $\C[N{]}$}

In \cite[Section 12]{L12}, Lusztig has given a geometric construction of the
enveloping algebra $U(\mathfrak{n})$ of the Lie algebra of $N$.
It is very similar to Ringel's realization of $U(\mathfrak{n})$
as the Hall algebra of $\md(\F_qQ)$ ``specialized at $q=1$''.
Here $Q$ denotes any quiver obtained by orienting the edges
of the Dynkin diagram $\Delta$.

There are two main differences between Ringel's and Lusztig's
constructions. First, 
in Lusztig's approach one works directly at $q=1$ by replacing
the counting measure for varieties over finite fields by the
Euler characteristic measure for constructible subsets of complex
algebraic varieties. The second difference is that one replaces 
the module varieties of $Q$ by the module varieties of $\La$,
in order to obtain a construction independent of the choice of
an orientation of $\Delta$.

As a result, one gets a model of $U(\mathfrak{n})$ in which
the homogeneous piece $U(\mathfrak{n})_\mathbf{d}$ of multidegree 
$\mathbf{d}=(d_i)\in\N^I$ is realized
as a certain vector space of complex-valued constructible functions   
on the variety $\La_{\mathbf{d}}$ of $\La$-modules with dimension
vector $\mathbf{d}$.
It follows that to every $M\in\md\La$ of dimension vector
$\mathbf{d}$, one can attach a natural element of the dual space 
$U(\mathfrak{n})_\mathbf{d}^*$, namely the linear form $\delta_M$
mapping a constructible function $\psi \in
U(\mathfrak{n})_\mathbf{d}$ to its evaluation at $M$ (by regarding
$M$ as a point of $\La_{\mathbf{d}}$).
Let
\[
U(\mathfrak{n})_{\rm gr}^*:=\bigoplus_{\mathbf{d}\in\N^I}
U(\mathfrak{n})_\mathbf{d}^*
\]
be the {\em graded} dual of $U(\mathfrak{n})$
endowed with the dual Hopf structure.
The following result is well-known to the experts,
but we were unable to find a convenient reference. 
We include a sketch of proof for the convenience of the reader.

\begin{proposition}
$U(\mathfrak{n})^*_{\rm gr}$ is isomorphic, as a Hopf algebra,
to $\C[N]$.
\end{proposition} 
\begin{proof} (Sketch.)
$H=U(\mathfrak{n})^*_{\rm gr}$ is a commutative Hopf algebra, and 
therefore it can be regarded as the coordinate ring of the
affine algebraic group $\Hom_{\rm alg}(H,\C)$ of algebra homomorphisms
from $H$ to $\C$, or equivalently as the coordinate ring
of the group $G(H^\circ)$ of all group-like
elements in the dual Hopf algebra $H^\circ$ (see \emph{e.g.}
\cite[\S3.4]{A}).
Note that $H$ being the graded dual of $U(\mathfrak{n})$, 
the dual $H^*$ of $H$ is the completion $\widehat{U(\mathfrak{\nn})}$
of $U(\mathfrak{n})$ with respect to its grading. 
A simple calculation shows that for every $e\in\mathfrak{n}$ 
the exponential 
$\exp(e) = \sum_{k\ge 0} e^k/k!\in \widehat{U(\mathfrak{\nn})}$ 
is a group-like element in $H^\circ$.
Let $e_i\ (i\in I)$ be the Chevalley generators of $\mathfrak{n}$.
Then the map $x_i(t) \mapsto \exp(te_i)\ (i\in I)$ extends 
to a homomorphism from $N$ to $G(H^\circ)$.
One can check that this is an isomorphism using the fact that 
$H$ is a polynomial algebra in $r$ variables. 
This induces the claimed Hopf algebra isomorphism from $H$
to $\C[N]$. 
\end{proof}

Let $\iota : U(\mathfrak{n})^*_{\rm gr} \to \C[N]$ denote
this isomorphism.
Let $\omega$ denote the automorphism of $\C[N]$ described
in \cite[\S 1.7]{GLSLMS}. It anti-commutes with the comultiplication,
and the corresponding anti-automorphism of $N$ leaves invariant
the one-parameter subgroups $x_i(t)$. 
In other words, for $f \in \C[N]$ we have
\[
(\omega f)(x_{i_1}(t_1)\cdots x_{i_k}(t_k)) =
f(x_{i_k}(t_k)\cdots x_{i_1}(t_1)),
\quad
(i_1,\ldots i_k \in I,\ t_1, \ldots,t_k \in \C).
\]
Define $\varphi_M = \omega\circ\iota(\delta_M)$.
We have thus obtained a map $M \mapsto \varphi_M$ from $\md\La$
to $\C[N]$. Let us describe it more explicitly.

Consider a composition series 
\[
\mathfrak{f}=(\{0\}=M_0 \subset M_1 \subset \cdots \subset M_d = M)
\]
of $M$ with simple factors $M_k/M_{k-1} \cong S_{i_k}$. 
We call $\ii:=(i_1,\ldots ,i_d)$ the \emph{type} of~$\mathfrak{f}$.
Let $\Phi_{\ii,M}$ denote the subset of the flag variety of $M$
(regarded as a $\C$-vector space) consisting of all flags which
are in fact composition series of $M$ (regarded as a $\La$-module)
of type $\ii$. This is a closed subset of the flag variety, hence
a projective variety. We denote by 
$\chi_{\ii,M} = \chi(\Phi_{\ii,M})\in\Z$ its Euler characteristic.
By unwinding Lusztig's construction of $U(\mathfrak{n})$,
dualizing it, and going through the above isomorphisms,
one gets the following formula for $\varphi_M$.
\begin{proposition}\label{prop51}
For every $\ii=(i_1,\ldots,i_k)\in I^k$ we have
\[ 
\varphi_M(x_{i_1}(t_1)\cdots x_{i_k}(t_k)) = 
\sum_{\ab\in\N^k}\chi_{\ii^\ab,M} 
\,\frac{t_1^{a_1}\cdots  t_k^{a_k}}{a_1!\cdots a_k!},
\]
where we use the short-hand notation  
$\ii^\ab=(\underbrace{i_1,\ldots,i_1}_{a_1},\ldots,
\underbrace{i_k,\ldots,i_k}_{a_k})$.   
\end{proposition}

\begin{proof} (Sketch.)
Using the above embedding of $N$ in $\widehat{U(\mathfrak{n})}$, we have
\[
x_{i_1}(t_1)\cdots x_{i_k}(t_k) =
\sum_{\ab\in\N^k} \frac{t_1^{a_1}\cdots t_k^{a_k}}{a_1!\cdots a_k!} \, 
e_{i_1}^{a_1}\cdots e_{i_k}^{a_k},
\]
as an element of $\widehat{U(\mathfrak{n})}$.  
Now, for a fixed $\jj=(j_1,\ldots,j_d)$, consider the constructible
function $\chi_\jj : M \mapsto \chi_{\jj,M}$ defined on $\La_\dd$,
where $\dd=(d_i)$ and $d_i$ is the number of $s$'s such that
$j_s=i$.
In Lusztig's Lagrangian construction of $U(\mathfrak{n})$ \cite{L12},
the functions $\chi_\jj$ span the vector space $U(\mathfrak{n})_\dd$.
More precisely, $\chi_\jj$ is identified with the monomial
$e_{j_d}\cdots e_{j_1}$.
By the definition of $\varphi_M$, we thus get
\[ 
\varphi_M(x_{i_1}(t_1)\cdots x_{i_k}(t_k))\equiv
\delta_M\left(\sum_{\ab\in\N^k} \frac{t_1^{a_1}\cdots t_k^{a_k}}{a_1!\cdots a_k!} \, 
\chi_{\ii^\ab}\right)
= \sum_{\ab\in\N^k}\chi_{\ii^\ab,M} 
\,\frac{t_1^{a_1}\cdots  t_k^{a_k}}{a_1!\cdots a_k!},
\]
as claimed.
Note that the twist by $\omega$ and the twist 
$\chi_{j_1,\ldots,j_d} \equiv e_{j_d}\cdots e_{j_1}$ cancel 
each other.
\end{proof}

\begin{remark}
{\rm
In \cite{GLSInv} we have denoted by $\varphi_M$ the function
$\iota(\delta_M)$, without twisting by $\omega$ 
(in the definition of $\Phi_{\ii,M}$
we were using
descending flags $\mathfrak{f}$ instead of ascending ones).
On the other hand, in \cite[\S 7]{GLSNag} we have defined
a left $U(\mathfrak{n})$-module structure on 
$U(\mathfrak{n})^*_{\rm gr}$.
The twisting by $\omega$ is needed if we want this structure
to agree with the usual left $U(\mathfrak{n})$-module structure on 
$\C[N]$ given by
\[
(e_i f)(x) = \frac{d}{dt} f(xx_i(t))\mid_{t=0},
\qquad
(f\in\C[N],\ x\in N).
\]
This is the convention which we have taken in \cite{GLSFour} and 
which we follow here.
}
\end{remark}

\begin{example}\label{exa52}
{\rm
In type $A_2$, we have $w_0=s_1s_2s_1$, hence every $f\in\C[N]$
is determined by its values at $x_1(t_1)x_2(t_2)x_1(t_3)$ for $(t_1,t_2,t_3)\in\C^3$.
One calculates 
\begin{align*}
x_1(t_1)x_2(t_2)x_1(t_3) &= 
\begin{pmatrix}
1&t_1&0\cr
0&1&0\cr
0&0&1
\end{pmatrix}
\begin{pmatrix}
1&0&0\cr
0&1&t_2\cr
0&0&1
\end{pmatrix}
\begin{pmatrix}
1&t_3&0\cr
0&1&0\cr
0&0&1
\end{pmatrix}
\\
&=
\begin{pmatrix}
1&t_1+t_3&t_1t_2\cr
0&1&t_2\cr
0&0&1
\end{pmatrix} .
\end{align*}
On the other hand, using the formula of Proposition~\ref{prop51} one
gets easily
\begin{align*}
\varphi_{S_1}(x_1(t_1)x_2(t_2)x_1(t_3))&=t_1+t_3,
\\
\varphi_{S_2}(x_1(t_1)x_2(t_2)x_1(t_3))&=t_2,
\\
\varphi_{Q_1}(x_1(t_1)x_2(t_2)x_1(t_3))&=t_1t_2,
\\
\varphi_{Q_2}(x_1(t_1)x_2(t_2)x_1(t_3))&=t_2t_3.
\end{align*}
It follows that, in terms of matrix entries, we have
\[
\varphi_{S_1} = n_{12},\quad
\varphi_{S_2} = n_{23},\quad
\varphi_{Q_1} = n_{13},\quad
\varphi_{Q_2} = 
\left|
\begin{matrix}
n_{12}&n_{13}\\
1&n_{23}
\end{matrix}
\right|.
\]     
}
\end{example}

\begin{exo}
{\rm
In type $A_n$, for $1\le i\le j \le n$, let $M_{[i,j]}$
denote the indecomposable
$\La$-module of dimension $j-i+1$ with socle $S_i$ and top $S_j$. 
($M_{[i,j]}$ is in fact uniserial.)
Show that $\varphi_{M_{[i,j]}}=n_{i,j+1}$, the matrix
entry on row $i$ and column $j+1$.
}
\end{exo}

\begin{exo}\label{exo54}
{\rm
In type $A_n$, 
show that $\varphi_{Q_k}$ is equal to the $k\times k$
minor of
\[
\begin{pmatrix}
1& n_{12}& \ldots & n_{1,n+1}\\
0& 1& \ldots & n_{2,n+1}\\
\vdots&\vdots&\ddots&\vdots \\
0&0&\ldots &1
\end{pmatrix}
\] 
with row indices $1,2,\ldots, k$ and column indices
$n-k+2, n-k+3, \ldots, n+1$.

More generally, show that for every submodule $M$ of
$Q_k$, $\varphi_M$ is equal to a $k\times k$ minor 
with row indices $1,2,\ldots, k$, and that conversely,
every nonzero $k\times k$ minor 
with row indices $1,2,\ldots, k$ is of the form 
$\varphi_M$ for a unique submodule $M$ of $Q_k$. 
}
\end{exo}

\begin{exo}\label{exo55}
{\rm
In type $D_4$ a reduced decomposition of $w_0$ is for
example 
\[
w_0 = s_1s_2s_4s_3s_1s_2s_4s_3s_1s_2s_4s_3.
\]
Using the realization of $N$ as a group of unitriangular
$8 \times 8$ matrices given in Example~\ref{exa42}, calculate
(possibly with the help of your favorite computer algebra system)
the product
\[
x=x_1(t_1)x_2(t_2)x_4(t_3)x_3(t_4)
x_1(t_5)x_2(t_6)x_4(t_7)x_3(t_8)
x_1(t_9)x_2(t_{10})x_4(t_{11})x_3(t_{12}).
\]

Check that the first row of the matrix $x$ is equal to
\[
[1,\qquad 
t_3 + t_7 + t_{11},\qquad 
t_3t_4+t_3t_8+t_7t_8+t_3t_{12}+t_7t_{12}+t_{11}t_{12},
\]
\[
t_3t_4t_6+t_3t_4t_{10}+t_3t_8t_{10}+t_7t_8t_{10},\qquad
t_3t_4t_5+t_3t_4t_9+t_3t_8t_9+t_7t_8t_9,
\]
\[
t_3t_4t_5t_6+t_3t_4t_5t_{10}+t_3t_4t_6t_9+t_3t_4t_9t_{10}
+t_3t_8t_9t_{10}+t_7t_8t_9t_{10},
\]
\[
t_3t_4t_5t_6t_8+t_3t_4t_5t_6t_{12}+t_3t_4t_6t_9t_{12}
+t_3t_4t_5t_{10}t_{12}+t_3t_4t_9t_{10}t_{12}
\]
\[
+\ t_3t_8t_9t_{10}t_{12}
+t_7t_8t_9t_{10}t_{12},\qquad
t_3t_4t_5t_6t_8t_{11}]
 .
\]

Check that the 8 entries on this row are equal to $\varphi_M(x)$
where $M$ runs over the 8 submodules of $Q_4$, including
the zero and the full submodules (see Example~\ref{exa31}).

Express in a similar way all the entries of $x$ as the evaluations
at $x$ of functions $\varphi_M$ where $M$ runs over the
subquotients of $Q_4$. 

Investigate the relations between the $2\times 2$ minors 
taken on the first 2 rows of $x$ and the values $\varphi_M(x)$ 
where $M$ is a submodule of $Q_3$. 
}
\end{exo}

\section{Multiplicative properties of $\varphi$}

In the geometric realization of $U(\mathfrak{n})$ given
in \cite{L12}, only the multiplication is constructed,
or equivalently the comultiplication of $\C[N]$.
For our purposes though, it is essential to study the
multiplicative properties of the maps $\varphi_M$.
The most important ones are 
\begin{theorem}[\cite{GLSAnnENS, GLSComp}]\label{th61}
Let $M, N \in \md\La$. Then the following hold:
\begin{enumerate}
\item $\varphi_{M}\varphi_{N} = \varphi_{M\oplus N}$.
\item Assume that $\dim \Ext_{\La}^1(M,N) = 1$. 
Let 
\[
0\rightarrow M \rightarrow X \rightarrow N\rightarrow 0,
\qquad 
0\rightarrow N \rightarrow Y \rightarrow M\rightarrow 0,
\] 
be two non-split short exact sequences (note that this determines 
$X$ and $Y$ uniquely up to isomorphism). Then
$
\varphi_{M}\varphi_{N} = \varphi_X + \varphi_Y.
$
\end{enumerate}
\end{theorem}
Note that in \cite{GLSComp} a formula is proved which generalizes
$(2)$ to any pair $(M,N)$ of $\La$-modules with 
$\dim \Ext^1_{\La}(M,N) > 0$. It involves all possible middle terms of
non-split short exact sequences with end terms $M$ and $N$, weighted
by certain Euler characteristics. 
It was inspired by a similar formula of Caldero and Keller 
in the framework of cluster categories \cite{CK}.
We will not need this general 
multiplication formula here.
\begin{example}
{\rm
Type $A_2$. 
Using the formulas of Example~\ref{exa52}, one checks easily that
\[
\varphi_{S_1}\varphi_{S_2} = \varphi_{Q_1} + \varphi_{Q_2},
\]
in agreement with the short exact sequences
\[
0\rightarrow S_1 \rightarrow Q_1 \rightarrow S_2 \rightarrow 0,
\qquad
0\rightarrow S_2 \rightarrow Q_2 \rightarrow S_1\rightarrow 0.
\]
}
\end{example}

\begin{exo}
{\rm
Type $A_3$. 
Consider the following indecomposable $\La$-modules
defined unambiguously by means of their socle filtration:
\[
M= S_2, \quad
N=\begin{tabular}{c}
$S_1 \oplus S_3$\\
\hline
$S_2$
\end{tabular}, \quad
X=
\begin{tabular}{c}
$S_2$\\
\hline
$S_1 \oplus S_3$\\
\hline
$S_2$
\end{tabular}=Q_2, \quad
Y=
\begin{tabular}{c}
$S_1$\\
\hline
$S_2$
\end{tabular}, \quad
Z=
\begin{tabular}{c}
$S_3$\\
\hline
$S_2$
\end{tabular}.
\]
Check that
$
\varphi_{M}\varphi_{N} = \varphi_{X} + \varphi_{Y\oplus Z}.
$
Using Exercise~\ref{exo54}, show that this identity is nothing
else than the classical Pl\"ucker relation
\[
[1,3]\times [2,4] =  [1,2]\times [3,4] + [1,4]\times [2,3]
\]
between $2\times2$ minors of the matrix of coordinate functions
\[
\begin{pmatrix}
1& n_{12}& n_{13} & n_{14} \\
0& 1& n_{23} & n_{24}\\
0& 0& 1& n_{34} \\
0& 0& 0& 1
\end{pmatrix},
\] 
where $[i,j]$ denotes the $2\times 2$ minor on rows $(1,2)$
and columns $(i,j)$. 
}
\end{exo}

\section{The dual semicanonical basis}
\label{sect7}

The functions $\varphi_M\ (M\in\md\La)$ satisfy many linear
relations. For example if $\dim \Ext_{\La}^1(M,N) = 1$, combining 
(1) and (2) in Theorem~\ref{th61} we get
\[
\varphi_{M\oplus N} = \varphi_X + \varphi_Y.
\]
It is possible, though, to form bases of $\C[N]$ consisting
of functions $\varphi_M$ where $M$ is taken in a 
certain restricted family of modules $M$.
For example, let $Q$ be a fixed orientation of $\Delta$.
Every $\C Q$-module can be regarded as a $\La$-module
in an obvious way. 
It is easy to check that 
\[
\{\varphi_M\mid M\in \md(\C Q)\}
\]
is a $\C$-basis of $\C[N]$. 
In fact this is the dual of the PBW-basis of $U(\mathfrak{n})$
constructed from $Q$ by Ringel (see \cite[\S 5.9]{GLSAnnENS}).

Unfortunately, this basis depends on the choice of the orientation
$Q$.
Using some geometry, one can obtain
a more ``canonical'' basis of $\C[N]$.
Let us fix a dimension vector $\dd\in\N^I$ and regard the 
map $\varphi$ as a map from the module variety $\La_{\dd}$
to $\C[N]$.
This is a constructible map, hence on every irreducible component
of $\La_{\dd}$ there is a Zariski open set on which $M\mapsto
\varphi_M$ is constant. Let us say that a module $M$ in this
open set is \emph{generic}.
Then, dualizing Lusztig's construction in \cite{L4}, one gets
\begin{theorem}
$\{\varphi_M\mid M \text{ is generic}\}$
is a basis of $\C[N]$.
\end{theorem}
This is the dual of Lusztig's \emph{semicanonical basis} of 
$U(\mathfrak{n})$.
We shall call it the \emph{dual semicanonical basis} of $\C[N]$. 
By construction it comes with a natural labelling by the 
union over all $\dd\in\N^I$ of the sets of irreducible
components of the varieties~$\La_\dd$. 

Important examples of generic modules are given by rigid modules.
We say that a $\La$-module $M$ is \emph{rigid} if 
$\Ext^1_\La(M,M)=0$, or equivalently if the orbit of $M$
in its module variety is open (see \cite{GLSInv}).

\begin{corollary}
If $M$ is a rigid $\La$-module then $\varphi_M$ belongs to
the dual semicanonical basis of $\C[N]$.
\end{corollary} 

The converse does not hold in general. 
More precisely every generic $\La$-module is rigid if and only
if $\La$ has type $A_n\ (n\le 4)$ (see \cite{GLSAnnENS}).

\begin{example}
{\rm
In type $D_4$, there is a one-parameter family of indecomposable
$\La$-modules with socle series and radical series
$$
\begin{tabular}{c}
$S_3$\\
\hline
$S_1 \oplus S_2 \oplus S_4$\\
\hline
$S_3$
\end{tabular}
$$
These modules are generic, but they are not rigid. 
For example there is a self-extension with middle term $Q_3$.
}
\end{example}

\section{Dual Verma modules} \label{sect8}

Let $\gf$ denote the Lie algebra of $G$ with its triangular 
decomposition $\gf = \nn\oplus\mathfrak{h}\oplus\nn_-$. 
Any $G$-module can also
be regarded as a $\gf$-module. 
We shall denote by $L(\la)$ the irreducible finite-dimensional 
module with highest weight $\la$.
It is conveniently constructed as the unique top factor of the
infinite-dimensional Verma $\gf$-module $M(\la)$
(see \emph{e.g.} \cite[I, \S 3.2]{CSM}).    
As a $U(\nn_-)$-module, $M(\la)$ is naturally isomorphic
to $U(\nn_-)$, hence we have a natural projection
$U(\nn)\cong U(\nn_-)\cong M(\la) \to L(\la)$ for every weight $\la$.
Dualizing and taking into account that $L(\la)$
is self-dual, we thus get an embedding 
$L(\la) \to M(\la)^* \cong\C[N]$.
This embedding has a nice description in terms of the
functions $\varphi_M$, as we shall now see.

Let $\la = \sum_{i\in I} a_i\varpi_i$ be the decomposition
of $\la$ in terms of the fundamental weights $\varpi_i$.
As $L(\la)$ is finite-dimensional, the $a_i$'s are nonnegative
integers.
Set $Q_\la = \oplus_{i\in I} Q_i^{\oplus a_i}$, an injective $\La$-module.

\begin{theorem}[\cite{GLSNag}]\label{th81}
In the above identification of  $M(\la)^*$ with $\C[N]$,
the irreducible representation $L(\la)$ gets identified with the
linear span of
\[
\{\varphi_M \mid M \text{ is a submodule of } Q_\la\}.
\] 
\end{theorem}
We refer to \cite{GLSNag} for an explicit formula calculating the images of
$\varphi_M \in L(\la)$ under the action of the Chevalley generators
$e_i$ and $f_i$ of $\gf$.

\begin{example}
{\rm
In type $A_n$, consider the fundamental representation 
$L(\varpi_k)$.
It is isomorphic to the
natural representation of $\SL(n+1,\C)$ in $\wedge^k\C^{n+1}$.
Using Exercise~\ref{exo54}, we recover via Theorem~\ref{th81}
that $L(\varpi_k)$ can
be identified with the subspace of $\C[N]$ spanned by the
$k\times k$-minors taken on the first $k$ rows of  
\[
\begin{pmatrix}
1& n_{12}& \ldots & n_{1,n+1}\\
0& 1& \ldots & n_{2,n+1}\\
\vdots&\vdots&\ddots&\vdots \\
0&0&\ldots &1
\end{pmatrix}
\] 
}
\end{example}

\begin{example}
{\rm
In type $D_4$, consider the fundamental representation $L(\varpi_4)$.
It is isomorphic to the defining representation of $G$ in $\C^8$.
If we
realize $N$ as a group of $8 \times 8$ unitriangular
matrices as in Example~\ref{exa42}
and use Exercise~\ref{exo55}, we recover via Theorem~\ref{th81} that $L(\varpi_4)$ can
be identified with the subspace of $\C[N]$ spanned by the
coordinate functions mapping an $8\times 8$ matrix
$x\in N$ to the entries of its first row.
}
\end{example}

\section{Parabolic subgroups and flag varieties}  
 
Let us fix a proper subset $K$ of $I$.
Denote by $y_i(t)\ (i\in I, t\in \C)$ 
the one-parameter subgroups of $G$ attached to
the negatives of the simple roots.
Let $B$ be the Borel subgroup of $G$ containing $N$.
The subgroup of $G$ generated by $B$ and the elements
$y_k(t)\ (k\in K, t\in \C)$ is called the standard
parabolic subgroup attached to $K$. We shall denote it
by $B_K$. In particular, $B_\emptyset = B$. 
It is known that every parabolic subgroup of $G$ is
conjugate to a standard parabolic subgroup.
The unipotent radical of $B_K$ will be denoted by $N_K$.
 In particular, $N_\emptyset = N$. 

\begin{example}\label{exa91}
{\rm
Let $G=\SL(5,\C)$, a group of type $A_4$. 
We choose for $B$ the subgroup of upper triangular 
matrices. 
Take $K=\{1,3,4\}$.
Then $B_K$ and $N_K$ are the subgroups of $G$ 
with the following block form:  
\[
B_K =   
\begin{pmatrix} * & * & * & * & *\cr
          * & * & * & * & * \cr
          0 & 0 & * & * & * \cr
          0 & 0 & * & * & * \cr
          0 & 0 & * & * & * 
\end{pmatrix},
\qquad
N_K =
\begin{pmatrix} 1 & 0 & * & * & *\cr
          0 & 1 & * & * & * \cr
          0 & 0 & 1 & 0 & 0 \cr
          0 & 0 & 0 & 1 & 0 \cr
          0 & 0 & 0 & 0 & 1 
\end{pmatrix}. 
\] 
}
\end{example}

Geometrically, $N_K$ is an affine space.
It can be identified with an open cell in the partial
flag variety $B_K^{-}\backslash G$, where $B_K^-$ is the
opposite parabolic subgroup (defined as $B_K$ but switching
the $x_i(t)$'s and the $y_i(t)$'s).
More precisely, the restriction to $N_K$ of the natural
projection $G \to B_K^{-}\backslash G$ is an open embedding.

\begin{example}\label{exa92}
{\rm
Let us continue Example~\ref{exa91}.
We have 
\[
B_K^{-} =   
\begin{pmatrix} * & * & 0 & 0 & 0\cr
          * & * & 0 & 0 & 0 \cr
          * & * & * & * & * \cr
          * & * & * & * & * \cr
          * & * & * & * & * 
\end{pmatrix}.
\]
Let $(u_1,\ldots,u_5)$ be the standard basis of $\C^5$.
We regard vectors of $\C^5$ as row vectors and
we let $G$ act on the right
on $\C^5$, so that the $k$th row of the matrix $g$
is $u_kg$.
Then $B_K^-$ is the stabilizer of the $2$-space
spanned by $u_1$ and $u_2$ for the induced 
transitive action of $G$ on the Grassmann
variety of $2$-planes of $\C^5$. Hence 
$B_K^-\backslash G$ is the Grassmannian $\Gr(2,5)$ 
of dimension~6. 

The unipotent subgroup
$N_K$ can be identified with the open subset of $\Gr(2,5)$
consisting of all the $2$-planes whose first Pl\"ucker
coordinate $[1, 2]$ does not vanish. 
}
\end{example}

\section{Coordinate rings}  \label{sect10}

Let $J$ be the complement of $K$ in $I$. 
The partial flag variety $B_K^-\backslash G$ can be 
naturally embedded as a closed subset in 
the product of projective spaces 
\[
\prod_{j\in J} \Pro(L(\varpi_j))
\]
\cite[p.123]{LG}. 
This is a generalization of the classical Pl\"ucker embedding
of the Grassmannian $\Gr(k,n+1)$ in 
$\Pro(\wedge^k\C^{n+1})=\Pro(L(\varpi_k))$.
We denote by $\C[B_K^-\backslash G]$ the multi-homoge\-neous
coordinate ring of $B_K^-\backslash G$ coming from this 
embedding. 
Let $\Pi_J\cong \N^J$ denote the monoid of dominant integral 
weights of the form
\[
\la = \sum_{j\in J} a_j \varpi_j, \qquad (a_j\in \N).
\]
Then, $\C[B_K^-\backslash G]$ is a $\Pi_J$-graded ring with a natural 
$G$-module structure. 
The homogeneous component 
with multi-degree $\la\in\Pi_J$ is an irreducible $G$-module 
with highest weight $\la$.
In other words, we have
\[
\C[B_K^-\backslash G] = \bigoplus_{\la\in\Pi_J} L(\la).
\]
Moreover, $\C[B_K^-\backslash G]$ is generated by its subspace
$\bigoplus_{j\in J} L(\varpi_j)$.

In particular, $\C[B^-\backslash G]=\bigoplus_{\la\in\Pi} L(\la)$,
where the sum is over the monoid $\Pi$ of all dominant integral
weights of $G$. This is equal to the affine coordinate ring
$\C[N^-\backslash G]$ of the multi-cone $N^-\backslash G$  
over $B^-\backslash G$, that is, to the ring 
\[
\C[N^-\backslash G]=\{f\in\C[G]\mid f(n\,g)=f(g),\ n\in N^-,\ g\in G\}
\]
of polynomial functions on $G$ invariant under $N^-$.
We will identify $\C[B_K^-\backslash G]$ with the subalgebra
of $\C[N^-\backslash G]$ generated by the homogeneous elements
of degree $\varpi_j\ (j\in J)$.
\begin{example}\label{exa101}
{\rm We continue Example~\ref{exa92}. 
The Pl\"ucker embedding of the Grassmannian $\Gr(2,5)$
consists in mapping the $2$-plane $V$ of $\C^5$  
with basis $(v_1,v_2)$ to the line spanned by $v_1\wedge v_2$ 
in $\wedge^2\C^5$, which is isomorphic to   
$L(\varpi_2)$.

This induces an embedding of $\Gr(2,5)$ into $\Pro(L(\varpi_2))$.
The homogeneous coordinate ring for this embedding is 
isomorphic to the subring of $\C[G]$ generated by the
functions $g\mapsto \De_{ij}(g)$, where 
$\De_{ij}(g)$ denotes the
$2\times 2$ minor of $g$ taken on columns 
$i,j$ and on the first two rows.
The $\De_{ij}$ are called Pl\"ucker coordinates. 
As a $G$-module we have
\[
\C[\Gr(2,5)]=\bigoplus_{k\in \N} L(k\varpi_2),   
\]
where the degree $k$ homogeneous component $L(k\varpi_2)$
consists of the homogeneous polynomials of degree $k$
in the Pl\"ucker coordinates.
}
\end{example}

Some distinguished elements of degree $\varpi_j$ in 
$\C[N^-\backslash G]$ are the generalized minors 
\[
\Delta_{\varpi_j, w(\varpi_j)},\qquad (w\in W),
\]
(see \cite[\S 1.4]{FZ}).
The image of $N_K$ in $B_K^-\backslash G$ under the natural
projection is the open subset defined by the non-vanishing
of the minors $\De_{\varpi_j,\varpi_j}\ (j\in J)$.
Therefore the affine coordinate ring of $\C[N_K]$
can be identified with the subring of degree $0$ homogeneous 
elements in the localized ring 
$\C[B_K^-\backslash G][\De_{\varpi_j,\varpi_j}^{-1},\,j\in J]$.
Equivalently, $\C[N_K]$ can be identified with the quotient
of $\C[B_K^-\backslash G]$ by the ideal generated by the 
elements $\De_{\varpi_j,\varpi_j}-1\ (j\in J)$.

\begin{example}\label{exa102}
{\rm  We continue Example~\ref{exa101}. 
The coordinate ring of $\C[N_K]$ is isomorphic to the 
ring generated by the $\De_{ij}$ modulo the relation
$\De_{12}=1$. This description may seem  
unnecessarily complicated since 
$N_K$ is just an affine space of dimension $6$ and 
we choose a presentation with $9$ 
generators and the Pl\"ucker relations. 
But these generators are better adapted to the cluster
algebra structure that we shall introduce. 
}
\end{example}

Let $\pr_J \colon \C[B_K^-\backslash G] \to \C[N_K]$ denote the 
projection obtained by modding out the ideal 
generated by the 
elements $\De_{\varpi_j,\varpi_j}-1\ (j\in J)$.
If $\C[B_K^-\backslash G]$ is identified as explained above
with a subalgebra of $\C[G]$, this map $\pr_J$ is nothing else
than the restriction of functions from $G$ to $N_K$.
The restriction of $\pr_J$ to
each homogeneous piece $L(\la)\ (\la\in\Pi_J)$ of 
$\C[B_K^-\backslash G]$ is injective and gives an embedding of 
$L(\la)$ into $\C[N_K]$. 
Moreover, any $f\in\C[N_K]$ belongs to $\pr_J(L(\la))$ for
some $\la\in\Pi_J$.

Summarizing this discussion and taking into account
Theorem~\ref{th81}, we get the following description
of $\C[N_K]$, convenient for our purpose.

\begin{theorem}[\cite{GLSFour}]\label{th103}
Let $R_J$ be the subspace of $\C[N]$ spanned by
\[
\{\varphi_M \mid M \text{ is a submodule of } Q_\la \text{ for some } \la\in\Pi_J\}.
\] 
The restriction to $R_J$ of the natural homomorphism 
$\C[N] \to \C[N_K]$ (given by restricting functions from $N$ to $N_K$)
is an isomorphism.
\end{theorem}

\section{The category $\Sub Q_J$}

Set $Q_J=\oplus_{j\in J} Q_j$.
Let $\Sub Q_J$ denote the full subcategory of $\md\La$ whose
objects are the submodules of a direct sum of a finite number
of copies of $Q_J$.
\begin{example}
{\rm 
Type $D_4$.
We have seen in Example~\ref{exa31} the structure of the 
indecomposable injective $Q_4$.
It is easy to see that $Q_4$ has only seven nonzero submodules
\[
\begin{tabular}{c}
$S_4$
\end{tabular}
,\quad
\begin{tabular}{c}
$S_3$\\
\hline
$S_4$
\end{tabular}
,\quad
\begin{tabular}{c}
$S_2$\\
\hline
$S_3$\\
\hline
$S_4$
\end{tabular}
,\quad
\begin{tabular}{c}
$S_1$\\
\hline
$S_3$\\
\hline
$S_4$
\end{tabular}
,\quad
\begin{tabular}{c}
$S_1 \oplus S_2$\\
\hline
$S_3$\\
\hline
$S_4$
\end{tabular}
,\quad
\begin{tabular}{c}
$S_3$\\
\hline
$S_1 \oplus S_2$\\
\hline
$S_3$\\
\hline
$S_4$
\end{tabular}
,\quad
\begin{tabular}{c}
$S_4$\\
\hline
$S_3$\\
\hline
$S_1 \oplus S_2$\\
\hline
$S_3$\\
\hline
$S_4$
\end{tabular},
\]
which are all indecomposable. 
It turns out that $\Sub Q_4$ contains a unique other indecomposable
object, which is a submodule of $Q_4\oplus Q_4$, and has the following
socle series
\[
\begin{tabular}{c}
$S_3$\\
\hline
$S_1 \oplus S_2$\\
\hline
$S_3$\\
\hline
$S_4 \oplus S_4$
\end{tabular}.
\]
Every object of $\Sub Q_4$ is a sum of copies of these eight
indecomposable objects.
}
\end{example}

Since $Q_J$ is an injective $\La$-module, the category $\Sub Q_J$ has
good homological properties \cite{AS}.
In particular, it is closed under extensions, 
has enough injectives, enough projectives and almost
split sequences. Moreover, the injectives coincide with the
projectives (it is a Frobenius category) and its stable
category is a 2-Calabi-Yau triangulated category.
Clearly, the algebra $R_J$ of Theorem~\ref{th103} is nothing
but the linear span of
\[
\{\varphi_M \mid M \in \Sub Q_J\}.
\]
Hence we may regard $\Sub Q_J$ as a kind of ``categorification''
of $\C[N_K]$. 
We are going to make this statement more precise by studying
the rigid modules in $\Sub Q_J$.
\begin{theorem}[\cite{GLSFour}]\label{th112}
Let $T$ be a rigid module in $\Sub Q_J$. The number of pairwise
nonisomorphic indecomposable direct summands of $T$ is
at most equal to $\dim N_K$.
\end{theorem}
In the case when $J=I$, that is, $\Sub Q_J = \md\La$, this
result was first obtained in \cite{GS}.

We shall say that a rigid module $T$ in $\Sub Q_J$ is \emph{complete}
if it has this maximal number of nonisomorphic summands.
Note that in this case, $T$ obviously contains the $n$ indecomposable
injective objects of $\Sub Q_J$.

In order to construct explicitly some complete rigid modules, we shall 
introduce certain functors.

\section{The functors $\E_i$ and $\E_i^\dag$}

For $i\in I$, we define an endo-functor $\E_i$ of $\md\La$ as follows.
Given an object $M\in\md\La$ we define $\E_i(M)$
as the kernel of the surjection 
$$
M \to S_i^{\oplus m_i(M)},
$$
where $m_i(M)$ denotes the multiplicity
of $S_i$ in the top of $M$.
If $f\colon M \to N$ is a homomorphism,  
$f(\E_i(M))$ is contained in $\E_i(N)$, and
we define 
$$
\E_i(f)\colon \E_i(M) \to \E_i(N)
$$ 
as the restriction of $f$ to
$\E_i(M)$.
Clearly, $\E_i$ is an additive functor.
It acts on a module $M$ by removing the $S_i$-isotypical
part of its top.
Similarly, one can define a functor 
$\E^\dag_i$ acting on $M$ by removing the 
$S_i$-isotypical part of its socle.
\begin{proposition}[\cite{GLSFour}]\label{prop121}
The functors $\E_i, \E_i^\dag\ (i\in I)$ satisfy the following relations:
\begin{itemize}
\item[(i)]
$\E_i\E_i = \E_i,\quad\E_i^\dag\E_i^\dag = \E_i^\dag$. 

\item[(ii)]  
$\E_i\E_j=\E_j\E_i, \quad 
\E_i^\dag\E_j^\dag=\E_j^\dag\E_i^\dag$, if $i$ and $j$ are not connected
by an edge in $\De$.

\item[(iii)] 
$\E_i\E_j\E_i = \E_j\E_i\E_j,\quad 
\E_i^\dag\E_j^\dag\E_i^\dag = \E_j^\dag\E_i^\dag\E_j^\dag$, if $i$ and $j$ are connected
by an edge.
\end{itemize}
\end{proposition}
Relations (ii) and (iii) are the braid relations for $\De$.
It follows that for any element $w$ of the Weyl group $W$ of $G$,
we have well-defined functors 
\[
\E_w:=\E_{i_1}\cdots\E_{i_k},\qquad \E^\dag_w:=\E^\dag_{i_1}\cdots\E^\dag_{i_k},
\]
where $w=s_{i_1}\cdots s_{i_k}$ is an arbitrary reduced 
decomposition of $w$.

Consider now the parabolic subgroup $W_K$ of $W$ generated by the 
$s_k\ (k\in K)$. This is a finite Coxeter group.
Let $w_0^K$ denote its longest element.
One can check that
$\E_{w_0^K}^\dag(M) \in \Sub Q_J$
for every $M\in \md\La$, and that $\E_{w_0^K}^\dag(M)=M$ 
if $M\in\Sub Q_J$.
In other words, the subcategory $\Sub Q_J$ can be described as
the image of $\md\La$ under the endo-functor $\E_{w_0^K}^\dag$ 
\cite{GLSFour}. 

\section{Construction of complete rigid modules}
\label{sect13}

The relevance of these functors for constructing rigid modules
comes from the following property.
\begin{proposition}[\cite{GLSFour}]
The functors $\E_w$ and $\E^\dag_w$ preserve rigid modules,
\emph{i.e.} if $M$ is rigid then $\E_w(M)$ and $\E^\dag_w(M)$
are also rigid.
\end{proposition} 

Let $w_0$ be the longest element of $W$, and let
$w_0=s_{i_1}\cdots s_{i_r}$ 
be a reduced decomposition
such that the first $r_K$ factors form a reduced decomposition
of $w_0^K$.
Set
\[
u_{\le p}=s_{i_1}\cdots s_{i_p}, \qquad
M_p=\E_{u_{\le p}}^\dag(Q_{i_p}), \qquad (p=1,\ldots,r).
\]
For $k\in K$, let $q_k = \max\{q\le r_K \mid i_q =k\}$.
Finally, define 
\[
T = M_{r_K+1}\oplus M_{r_K+2}\oplus \cdots \oplus M_r 
\oplus \left(\oplus_{k\in K} M_{q_k}\right)\oplus Q_J.
\]
\begin{theorem}[\cite{GLSFour}]\label{th123}
$T$ is a complete rigid module in $\Sub Q_J$.
\end{theorem}

Note that by construction the modules $M_p$ with $p>r_K$
are in the image of the functor $\E^\dag_{w_0^K}$, hence in 
$\Sub Q_J$.
Note also that $M_{q_k}=\E^\dag_{w_0^K}(Q_k)$ for $k\in K$.
The modules $M_{q_k}$ together with the summands of $Q_J$
are the indecomposable injectives of $\Sub Q_J$.
Finally, if $t_l = \max\{t\le r \mid i_t =l\}$,
then $M_{t_l}= \E_{w_0}^\dag(Q_l)=0$ for every $l\in I$.
It follows that $T$ has in fact 
$
r-r_K+|K|+|J|-|I| = r-r_K = \dim N_K
$
indecomposable summands, in agreement with Theorem~\ref{th112}.

\begin{example}\label{ex133}
{\rm
Type $D_4$. We take again $K=\{1,2,3\},\ J=\{4\}$ (remember that
the central node of $\De$ is labelled by 3).
Here $r_K=6$.
We choose the reduced decomposition
\[
w_0 = s_1s_3s_1s_2s_3s_1s_4s_3s_1s_2s_3s_4.
\]
We then have
\[
M_{s_2}=M_4 =
\begin{tabular}{c}
$S_2$\\
\hline
$S_3$\\
\hline
$S_4$
\end{tabular}
,\quad
M_{s_3}=M_5 =
\begin{tabular}{c}
$S_3$\\
\hline
$S_1 \oplus S_2$\\
\hline
$S_3$\\
\hline
$S_4 \oplus S_4$
\end{tabular}
,\quad
M_{s_1}=M_6 =
\begin{tabular}{c}
$S_1$\\
\hline
$S_3$\\
\hline
$S_4$
\end{tabular}
,\quad
\] 
\[
M_7 = S_4
,\quad
M_8 =
\begin{tabular}{c}
$S_3$\\
\hline
$S_4$
\end{tabular}
,\quad
M_9=M_{10}=M_{11}=M_{12}=0.
\]
The module $T=M_4\oplus M_5\oplus M_6\oplus M_7\oplus M_8\oplus Q_4$
is complete rigid in $\Sub Q_4$.
}
\end{example}

\section{Cluster algebras of geometric type}
\label{newsect}

Our next aim will be to associate to the category
$\Sub Q_J$ certain cluster algebras of geometric
type. We refer the reader to \cite{FZ1,FZ2,BFZ} for a
detailed exposition of their properties and of
the motivating examples of coordinate rings of double
Bruhat cells.
Here we shall merely recall their definition.

Let $d$ and $n$ be integers with $d\ge n\ge 0$.
If $B = (b_{ij})$ is any $d \times (d-n)$-matrix
with integer entries, then
the {\it principal part} of
$B$ is obtained 
by deleting from $B$ the last $n$ rows.
Given some $k \in [1,d-n]$ define a new $d \times (d-n)$-matrix 
$\mu_k(B) = (b_{ij}')$ by
$$
b_{ij}' =
\begin{cases}
-b_{ij} & \text{if $i=k$ or $j=k$},\\
b_{ij} + \dfrac{|b_{ik}|b_{kj} + b_{ik}|b_{kj}|}{2} & \text{otherwise},
\end{cases}
$$
where $i \in [1,d]$ and $j \in [1,d-n]$.
One calls $\mu_k(B)$ the \emph{mutation} 
of $B$ in direction~$k$.
If $B$ is an integer matrix whose principal part is
skew-symmetric, then it is 
easy to check that $\mu_k(B)$ is also an integer matrix 
with skew-symmetric principal part.
In this case, Fomin and Zelevinsky define a cluster algebra
${\mathcal A}(B)$ as follows.
Let ${\mathcal F} = \C(y_1,\ldots,y_d)$ be the field of rational
functions in $d$ commuting variables 
$\yy = (y_1,\ldots,y_d)$. 
One calls $(\yy,B)$ the {\it initial seed} 
of
${\mathcal A}(B)$.
For $1 \le k \le d-n$ define 
\begin{equation}\label{mutationformula}
y_k^* = 
\frac{\prod_{b_{ik}> 0} y_i^{b_{ik}} + \prod_{b_{ik}< 0} y_i^{-b_{ik}}}{y_k}.
\end{equation}
Let $\mu_k(\yy)$ denote the $d$-tuple obtained from $\yy$ by replacing 
$y_k$ by $y_k^*$. 
The pair 
$
(\mu_k(\yy),\mu_k(B))
$ 
is the \emph{mutation of the seed}
$(\yy,B)$ in direction $k$. 

Now one can iterate this process and mutate again each seed
$
(\mu_k(\yy),\mu_k(B))
$ 
in $d-n$ directions.
The collection of all seeds obtained from the initial seed
$(\yy,B)$ via a finite sequence of mutations is called the
\emph{mutation class} of $(\yy,B)$.
Each seed in this class consists of a $d$-tuple of algebraically 
independent elements of ${\mathcal F}$
called a 
\emph{cluster} 
and of a matrix called the 
\emph{exchange matrix}.
The elements of a cluster are its 
\emph{cluster variables}.
One does not mutate the last $n$ elements of a cluster. 
They are called
\emph{coefficients} and belong to every cluster.
The \emph{cluster algebra}
$
{\mathcal A}(B)
$
is by definition the subalgebra of ${\mathcal F}$ generated by the
set of all the cluster variables appearing in a seed
of the mutation class of $(\yy, B)$.
The subring of ${\mathcal A}(B)$ generated by the coefficients
is called the \emph{coefficient ring}.
The integer $d-n$ is called the \emph{rank} of ${\mathcal A}(B)$.
A monomial in the cluster variables is called a \emph{cluster
monomial} if all its variables belong to a single cluster.

\begin{example}
{\rm
Take $d=7$ and $n=5$. 
Let 
\[
B
=
\left[
\begin{matrix}
0&-1\\
1&0\\
-1&0\\
1&0\\
-1&1\\
0&-1\\
0&1
\end{matrix}
\right].
\]
Then the mutation in direction 1 reads
\[
\mu_1(B)
=
\left[
\begin{matrix}
0&1\\
-1&0\\
1&-1\\
-1&0\\
1&0\\
0&-1\\
0&1
\end{matrix}
\right],
\qquad
y_1^*=\frac{y_2y_4+y_3y_5}{y_1}.
\]
In this simple example, it turns out that ${\mathcal A}(B)$ 
has only a finite number of cluster variables.
In fact ${\mathcal A}(B)$ is isomorphic to the homogeneous
coordinate ring of the Grassmannian ${\rm Gr}(2,5)$ of 
$2$-planes of $\C^5$ \cite[\S 12]{FZ2}. The explicit isomorphism maps
the cluster variables $y_1, \ldots, y_7$ to the following
Pl\"ucker coordinates:
\[
y_1\mapsto [1,3],\,
y_2\mapsto [1,4],\,
y_3\mapsto [1,2],\,
y_4\mapsto [2,3],\,
y_5\mapsto [3,4],\,
y_6\mapsto [4,5],\,
y_7\mapsto [1,5].
\]
The other cluster variables obtained by mutation from this
initial seed are the remaining Pl\"ucker coordinates
$[2,4],\,[2,5],\,[3,5]$.
}
\end{example}

\section{Mutation of complete rigid modules}

We shall now introduce an operation of mutation for
complete rigid modules in $\Sub Q_J$, inspired by
the cluster algebra mutation of Fomin and Zelevinsky.

Let $T=T_1\oplus \cdots \oplus T_d$ be an arbitrary 
basic complete rigid module in $\Sub Q_J$.
Thus the $T_i$'s are indecomposable and pairwise non isomorphic,
and 
\[
d=r-r_K=\ell(w_0)-\ell(w_0^K)=\ell(w_0^Kw_0).
\]
Assume that the injective summands of $T$ are the last
$n$ ones, namely $T_{d-n+1}, \ldots, T_d$.
Relying on the results of \cite{GLSInv} we show in
\cite{GLSFour} the following

\begin{theorem}\label{th141}
Let $k\le d-n$.
There exists a unique indecomposable module $T_k^*\not\cong T_k$ in 
$\Sub Q_J$ such that $(T/T_k)\oplus T_k^*$ is a basic complete rigid 
module in $\Sub Q_J$.
Moreover, $\dim \Ext^1_\La(T_k,T_k^*) = 1$ and if 
\[
0 \to T_k \stackrel{g}{\to} X_k \stackrel{f}{\to} T^*_k \to 0,\qquad
0 \to T_k^* \stackrel{i}{\to} Y_k \stackrel{h}{\to} T_k \to 0
\]
are the unique non-split short exact sequences between $T_k$ and $T_k^*$,
then $f, g, h, i$ are minimal $\add(T/T_k)$-approximations,
and $X_k$ and $Y_k$ have no isomorphic indecomposable summands.
\end{theorem}

In this situation, we say that 
$(T/T_k)\oplus T_k^*$ is the {\em mutation of $T$ in the
direction of $T_k$}, and we write $\mu_k(T)=(T/T_k)\oplus T_k^*$.
Since $X_k$ and $Y_k$ belong to $\add(T)$, we can describe this
mutation by means of the multiplicities of each indecomposable
summand of $T$ in $X_k$ and $Y_k$.
This leads to associate to $T$ a matrix of integers
called its exchange matrix encoding the mutations of $T$ in
all possible $d-n$ directions.

More precisely, define
$b_{ik}=[X_k : T_i]$ if $X_k$ has summands isomorphic to
$T_i$, $b_{ik}=-[Y_k : T_i]$ if $Y_k$ has summands isomorphic to
$T_i$, and $b_{ik}=0$ otherwise. 
Note that these conditions are disjoint because
$X_k$ and $Y_k$ have no isomorphic direct summands.
The 
$d\times (d-n)$ matrix 
$B(T)=[b_{ik}]$
is called the \emph{exchange matrix} of~$T$.
We can now state:
\begin{theorem}[\cite{GLSInv,GLSFour}]
Let $T=T_1\oplus \cdots \oplus T_d$ be a complete rigid module
in $\Sub Q_J$ as above, and let $k\le d-n$. Then 
\[
B(\mu_k(T)) = \mu_k(B(T)),
\]
where on the right-hand side $\mu_k$ stands for the Fomin-Zelevinsky
matrix mutation. 
\end{theorem}
In other words, our mutation of complete rigid modules induces at the
level of exchange matrices the Fomin-Zelevinsky matrix mutation.

\section{Cluster algebra structure on $\C[N_K{]}$}
\label{sect15}

Let $T$ be one of the complete rigid modules of \S\ref{sect13}. 
Consider the mutation class $\R$ of $T$, that is, the
set of all complete rigid modules of $\Sub Q_J$ which can be obtained
from $T$ by a finite sequence of mutations. One can show that $\R$
contains all the rigid modules of \S\ref{sect13} corresponding to all possible
choices of a reduced decomposition of $w_0^K w_0$, hence $\R$ does not
depend on the choice of $T$.

We can now project $\R$
on $R_J\cong\C[N_K]$ using the map $M\mapsto \varphi_M$.
More precisely, for
$U=U_1\oplus \cdots \oplus U_d\in\R$,
let $x(U)=(\varphi_{U_1},\ldots,\varphi_{U_d})$,
(where again $d=r-r_K$).
The next result
follows from Theorem~\ref{th141} and Theorem~\ref{th61}.

\begin{theorem}
\begin{itemize}
\item[(i)]
$\{x(U)\mid U\in \R\}$ is the set of clusters
of a cluster algebra ${\cal A}_J\subseteq R_J\cong\C[N_K]$
of rank $d-n$.
\item[(ii)]
The coefficient ring of ${\cal A}_J$ is the
ring of polynomials in the $n$ variables $\varphi_{L_i}\ (i\in I)$,
where the $L_i$ are the indecomposable injective objects
in $\Sub Q_J$.
\item[(iii)]
All the cluster monomials belong to the dual semicanonical
basis of $\C[N]$, and are thus linearly independent.
\end{itemize}
\end{theorem}

The rigid modules $T$ of \S\ref{sect13} project to initial seeds
of the cluster algebra ${\cal A}_J$ that we are going to
describe.

For $i\in I$ and $u,v\in W$, let $\De_{u(\varpi_i),v(\varpi_i)}$
denote the generalized minor introduced by Fomin and 
Zelevinsky \cite[\S 1.4]{FZ}.
This is a regular function on $G$. 
We shall mainly work with the restriction of this 
function to $N$, that we shall denote by $D_{u(\varpi_i),v(\varpi_i)}$.
It is easy to see that $D_{u(\varpi_i),v(\varpi_i)}=0$
unless $u(\varpi_i)$ is less or equal to $v(\varpi_i)$ in the Bruhat 
order, and that $D_{u(\varpi_i),u(\varpi_i)}=1$ for
every $i\in I$ and $u\in W$.
It is also well known that $D_{\varpi_i,w_0(\varpi_i)}$ is a lowest
weight vector of $L(\varpi_i)$ in its realization as a subspace 
of $\C[N]$ explained in \S\ref{sect8}.
Therefore, using Theorem~\ref{th81}, we get 
\[
\varphi_{Q_i}=D_{\varpi_i,w_0(\varpi_i)},\qquad (i\in I).
\]
More generally, it follows from \cite[Lemma 5.4]{GLSLMS}
that for $u, v\in W$ we have
\[
\varphi_{\E^\dag_u\E_v Q_i} = D_{u(\varpi_i),vw_0(\varpi_i)}.
\]
Thus the elements of $\C[N]$ attached to the summands
$M_p$ of the complete rigid module $T$ of
Theorem~\ref{th123} are given by
\begin{align*}
\varphi_{M_p}
&=D_{\varpi_{i_p},u_{\le p}w_0(\varpi_{i_p})},\ 
\qquad (p\in \{r_K+1,\ldots,r\} \cup \{s_k \mid k\in K\}),
\\
\varphi_{Q_j}&=D_{\varpi_j,w_0(\varpi_j)},\qquad \qquad \ \, (j\in J).
\end{align*}
Moreover the matrix $B(T)$ can also be described explicitly
by means of a graph similar to those arising in the 
Chamber Ansatz of Fomin and Zelevinsky (see \cite[\S 9.3]{GLSFour}). 

\begin{example}\label{ex15.2}
{\rm
We continue Example~\ref{ex133}. We have 
\begin{align*}
\varphi_{M_4}&= D_{\varpi_2,s_1s_3s_1s_2w_0(\varpi_2)},
\qquad\qquad\,
\varphi_{M_5} = D_{\varpi_3,s_1s_3s_1s_2s_3w_0(\varpi_3)},
\\
\varphi_{M_6}&= D_{\varpi_1,s_1s_3s_1s_2s_3s_1w_0(\varpi_1)},
\qquad\ \ 
\varphi_{M_7} = D_{\varpi_4,s_1s_3s_1s_2s_3s_1s_4w_0(\varpi_4)},
\\
\varphi_{M_8}&= D_{\varpi_3,s_1s_3s_1s_2s_3s_1s_3w_0(\varpi_3)},
\qquad
\varphi_{Q_4} = D_{\varpi_4,w_0(\varpi_4)}.
\end{align*}
It turns out that in the matrix realization of $N$ given in
Example~\ref{exa42} the generalized minors above can be expressed
as ordinary minors of the unitriangular $8\times 8$ matrix $x\in N$.
Indeed, denoting the matrix entries of $x$ by $n_{ij}(x)$
one can check that
\begin{align*}
\varphi_{M_4}&= n_{14},
\qquad
\varphi_{M_5} = 
\left|
\begin{matrix}n_{17}&n_{18}\\1 & n_{78}
\end{matrix}
\right|,
\qquad
\varphi_{M_6} = n_{15},
\\
\varphi_{M_7}&= n_{12},
\qquad
\varphi_{M_8} = n_{13},
\qquad\qquad\ \ \ \
\varphi_{Q_4} = n_{18}.
\end{align*}
The cluster variables of this seed are $\varphi_{M_7}$
and $\varphi_{M_8}$. 
The exchange relations come from the following exact
sequences 
\[
0\to M_7 \to M_5 \to M_7^*\to 0,\qquad
0\to M_7^* \to Q_4 \to M_7\to 0,
\]
\[
0\to M_8 \to M_4 \oplus M_6 \to M_8^*\to 0,\qquad
0\to M_8^* \to M_5 \to M_8\to 0,
\]
where 
\[
M_7^*=\begin{tabular}{c}
$S_3$\\
\hline
$S_1 \oplus S_2$\\
\hline
$S_3$\\
\hline
$S_4$
\end{tabular}
\qquad
M_8^*=\begin{tabular}{c}
$S_1 \oplus S_2$\\
\hline
$S_3$\\
\hline
$S_4$
\end{tabular}.
\]
The exchange matrix is therefore
\[
B(T)
=
\left[
\begin{matrix}
0&0\\
0&0\\
0&-1\\
-1&1\\
0&-1\\
1&0
\end{matrix}
\right]
\]
where the rows are labelled by $(M_7, M_8, M_4, M_5, M_6, Q_4)$
and the columns by $(M_7, M_8)$.
}
\end{example}

A priori, we only have an inclusion of our cluster algebra 
${\cal A}_J$ in $R_J \cong \C[N_K]$, but we believe that
 \begin{conjecture}\label{conj1}
We have ${\cal A}_J = R_J$.
\end{conjecture}
The conjecture is proved for $G$ of type $A_n$ and of type $D_4$.
It is also proved for
$J=\{n\}$ in type $D_n$, and for $J=\{1\}$
in type $D_5$ (see \cite{GLSFour}). 
Moreover it follows from \cite{GLSAst} that it is also
true whenever $w_0^Kw_0$ has a reduced expression adapted to an 
orientation of the Dynkin diagram.


\section{Cluster algebra structure on $\C[B_K^-\backslash G{]}$}
\label{sect16}

Let us start by some simple remarks.
Consider the affine space $\C^r$ and the
projective space $\Pro(\C^{r+1})$.
The coordinate ring of $\C^r$ is $R=\C[x_1,\ldots,x_r]$, and
the homogeneous coordinate ring of $\Pro(\C^{r+1})$ is
$S=\C[x_1,\ldots,x_{r+1}]$.
Moreover $\C^r$ can be regarded as the open subset
of $\Pro(\C^{r+1})$ given by the non-vanishing of $x_{r+1}$.
Given a hypersurface $\Sigma\subset\C^r$ of equation 
$f(x_1,\ldots,x_r)=0$ for some $f\in R$, its completion
$\widehat{\Sigma}\subset \Pro(\C^{r+1})$ is described by
the equation $\widehat{f}(x_1,\ldots,x_{r+1})=0$, where
$\widehat{f}$ is the homogeneous element of $S$ obtained
by multiplying each monomial of $f$ by an appropriate power
of $x_{r+1}$.

Similarly, consider the open embedding 
$N_K \subset B_K^-\backslash G$ given by restricting 
the natural projection $G \to B_K^-\backslash G$ to $N_K$.
By this embedding, $N_K$ is identified with the subset of
$B_K^-\backslash G$ given by the simultaneous non-vanishing 
of the generalized minors $\De_{\varpi_j,\varpi_j}\ (j\in J)$.
To an element $f\in \C[N_K]$ we can associate a 
$\Pi_J$-homogeneous element
$\widetilde{f}\in \C[B_K^-\backslash G]$ by multiplying each
monomial in $f$ by an appropriate monomial in the 
$\De_{\varpi_j,\varpi_j}$'s. 
More precisely, using the notation of \S\ref{sect10},
$\widetilde{f}$ is the homogeneous element of $\C[B_K^-\backslash G]$
with smallest degree such that $\pr_J(\widetilde{f})=f$.

With this notation, we can now state the following result
of \cite{GLSFour}.
\begin{theorem}
\begin{itemize}
\item[(i)]
$\{\widetilde{x(U)}\mid U\in \R\}$ is the set of clusters
of a cluster algebra $\widetilde{\cal A}_J\subseteq \C[B_K^-\backslash G]$
of rank $d-n$.
\item[(ii)]
The coefficient ring of $\widetilde{\cal A}_J$ is the
ring of polynomials in the $n+|J|$ variables 
$\widetilde{\varphi_{L_i}}\ (i\in I)$
and $\De_{\varpi_j,\varpi_j}\ (j\in J)$.
\item[(iii)]
The exchange matrix $\widetilde{B}$ attached to $\widetilde{x(U)}$
is obtained from the exchange matrix $B$ of $x(U)$ by
adding $|J|$ new rows (in the non-principal part)
labelled by $j\in J$, where the entry in column $k$ and row $j$ is
equal to 
\[
b_{jk}=\dim\Hom_\La(S_j,X_k)-\dim \Hom_\La(S_j,Y_k).
\]
Here, if $U_k$ denotes the $k$th summand of $U$, $U_k^*$ its
mutation, then $X_k, Y_k$ are the middle terms of the non-split
short exact sequences
\[
0 \to U_k \to X_k \to U^*_k \to 0,\qquad
0 \to U_k^* \to Y_k \to U_k \to 0.
\]
\end{itemize}
\end{theorem}

\begin{example}
{\rm
We continue Example~\ref{ex133} and Example~\ref{ex15.2}.
So we are in type $D_4$ with $K=\{1,2,3\}$ and $J=\{4\}$.
The cluster $\widetilde{x(T)}$ consists of the $7$ functions
\[
\widetilde{\varphi_{M_4}},\ \widetilde{\varphi_{M_5}},\ 
\widetilde{\varphi_{M_6}},\ \widetilde{\varphi_{M_7}},\ 
\widetilde{\varphi_{M_8}},\ \widetilde{\varphi_{Q_4}},\ 
\Delta_{\varpi_4,\varpi_4}.
\]
The exchange matrix $\widetilde{B(T)}$ is obtained from the
matrix $B(T)$ of Example~\ref{ex15.2} by adding a new row labelled
by the extra coefficient $\Delta_{\varpi_4,\varpi_4}$.
Since
\[
\dim\Hom_\La(S_4,M_5)-\dim \Hom_\La(S_4,Q_4)=2-1=1,
\]
\[
\dim\Hom_\La(S_4,M_4\oplus M_6)-\dim \Hom_\La(S_4,M_5)=2-2=0,
\]
this new row is equal to $[1,0]$, thus 
\[
\widetilde{B(T)}
=
\left[
\begin{matrix}
0&0\\
0&0\\
0&-1\\
-1&1\\
0&-1\\
1&0 \\
1&0
\end{matrix}
\right].
\]
Note that in this example, the variety $B_K^-\backslash G$ is isomorphic
to a smooth quadric in $\Pro(\C^8)$.
Its homogeneous coordinate ring $\C[B_K^-\backslash G]$ coincides with
the affine coordinate ring of the isotropic cone in $\C^8$ of the 
corresponding non-degenerate quadratic form.
Thus we recover an example of \S~\ref{sect2}.
The precise identification is via the following formulas
(see Exercise~\ref{exo55}):
\[
y_1=\De_{\varpi_4,\varpi_4},\ 
y_2=\widetilde{\varphi_{M_7}},\ 
y_3=\widetilde{\varphi_{M_8}},\
y_4=\widetilde{\varphi_{M_4}},\
y_5=\widetilde{\varphi_{M_6}},\
\]
\[
y_6=\widetilde{\varphi_{M_8^*}},\
y_7=\widetilde{\varphi_{M_7^*}},\
y_8=\widetilde{\varphi_{Q_4}},\
p=\widetilde{\varphi_{M_5}}.
\]
Note that since $\C[B_K^-\backslash G]$ is generated by the
$y_i\ (1\le i\le 8)$, which are cluster variables or generators
of the coefficient ring of $\widetilde{\cal A}_J$, we have
in this case that $\widetilde{\cal A}_J=\C[B_K^-\backslash G]$.
} 
\end{example}
When $J=\{j\}$ and $G$ is of type $A$, $B_K^-\backslash G$
is a Grassmannian and the cluster algebra $\widetilde{\cal A}_J$
coincides with the one defined by Scott in \cite{S}.

When $K=\emptyset$, $\C[B_K^-\backslash G] = \C[N^-\backslash G]$
and $J=I$ the cluster algebra $\widetilde{\cal A}_J$ is essentially the
same as the one attached by Berenstein, Fomin and Zelevinsky 
to the big cell of the base affine space $N^-\backslash G$
in \cite[\S 2.6]{BFZ}.
More precisely, both cluster algebras have identical seeds,
but Berenstein, Fomin and Zelevinsky consider an
{\em upper} cluster algebra, and they assume that the 
coefficients 
\[
\widetilde{\varphi_{L_i}}=\widetilde{\varphi_{Q_i}}=\De_{\varpi_i,w_0(\varpi_i)},
\quad
\De_{\varpi_i,\varpi_i},\quad (i\in I),
\] 
are invertible,
{\em i.e.} the ring of coefficients consists of Laurent polynomials. 

Let $\Sigma_J$ be the multiplicative submonoid of $\widetilde{\cal A}_J$
generated by the set
\[
\{\De_{\varpi_j,\varpi_j} \mid j\in J \mbox{ and } \varpi_j
\mbox{ is {\em not} a minuscule weight}\}.
\]

\begin{conjecture}\label{conj2}
The localizations of $\widetilde{\cal A}_J$ and 
$\C[B_K^-\backslash G]$ with respect to $\Sigma_J$ are equal.
\end{conjecture}
Note that if $J$ is such that all the weights
$\varpi_j\ (j\in J)$ are minuscule, then $\Sigma_J$ is trivial
and the conjecture states that the algebras $\widetilde{\cal A}_J$ and 
$\C[B_K^-\backslash G]$ coincide without localization.
This is in particular the case for every $J$ in type $A_n$.

The conjecture is proved for $G$ of type $A_n$ and of type $D_4$.
It is also proved for
$J=\{n\}$ in type $D_n$, and for $J=\{1\}$
in type $D_5$ (see \cite{GLSFour}). 
Note also that Conjecture~\ref{conj2} implies 
Conjecture~\ref{conj1}.


\section{Finite type classification}


Recall that a cluster algebra is said to be of {\em finite
type} if it has finitely many cluster variables, or
equivalently, finitely many clusters. 
Fomin and Zelevinsky have classified the cluster algebras
of finite type \cite{FZ2}, attaching to them a finite
root system called their {\em cluster type}.

Note that the clusters of ${\cal A}_J$ and  $\widetilde{\cal A}_J$
are in natural one-to-one correspondence, and that the principal
parts of the exchange matrices of two corresponding clusters
are the same. This shows that ${\cal A}_J$ and $\widetilde{\cal A}_J$
have the same cluster type, finite or infinite.

\begin{table}
\begin{center}
\begin{tabular}
{|c|c|c|}
\hline
Type of $G$ & $J$ & Type of ${\cal A}_J$\\
\hline
$A_n$ $(n\ge 2)$& $\{1\}$ & --- \\
$A_n$ $(n\ge 2)$& $\{2\}$ & $A_{n-2}$ \\
$A_n$ $(n\ge 2)$& $\{1,2\}$ & $A_{n-1}$\\
$A_n$ $(n\ge 2)$& $\{1,n\}$ & $(A_1)^{n-1}$\\
$A_n$ $(n\ge 3)$& $\{1,n-1\}$ & $A_{2n-4}$ \\
$A_n$ $(n\ge 3)$& $\{1,2,n\}$ & $A_{2n-3}$ \\
\hline
$A_4$ & $\{2,3\}$ & $D_4$\\
$A_4$ & $\{1,2,3\}$ & $D_5$\\
$A_4$ & $\{1,2,3,4\}$ & $D_6$\\
\hline
$A_5$ & $\{3\}$ & $D_4$ \\
$A_5$ & $\{1,3\}$ & $E_6$\\
$A_5$ & $\{2,3\}$ & $E_6$\\
$A_5$ & $\{1,2,3\}$ & $E_7$\\
\hline
$A_6$ & $\{3\}$ & $E_6$ \\
$A_6$ & $\{2,3\}$ & $E_8$\\
\hline
$A_7$ & $\{3\}$ & $E_8$ \\
\hline
$D_n$ $(n\ge 4)$ & $\{n\}$ & $(A_1)^{n-2}$\\
\hline
$D_4$ & $\{1,2\}$ & $A_5$\\
\hline
$D_5$ & $\{1\}$ & $A_5$\\
\hline
\end{tabular}
\end{center}
\caption{\small \it Algebras ${\cal A}_J$ of finite cluster type.
\label{tab1}}
\end{table} 

Using the explicit initial seed described in \S\ref{sect15}
it is possible to give a complete list of the
algebras ${\cal A}_J$ which have a finite cluster type \cite{GLSFour}.
The results are summarized in Table~\ref{tab1}.
Here, we label the vertices of the Dynkin diagram of type
$D_n$ as follows:
\input epsf.tex
\begin{center}
\leavevmode
\epsfxsize =6cm
\epsffile{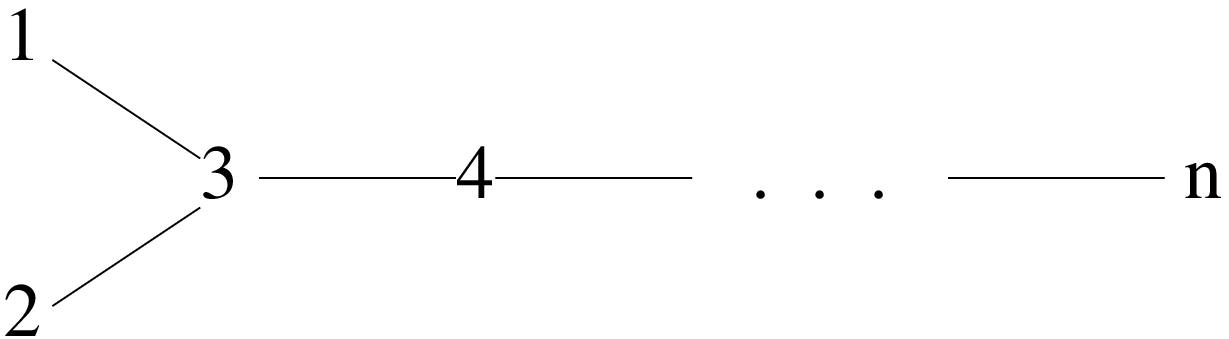}
\end{center}
We have only listed one representative of each orbit under a
diagram automorphism. For example, in type $A_n$
we have an order $2$ diagram automorphism mapping
$J=\{1,2\}$ to $J'=\{n-1,n\}$. Clearly,
${\cal A}_{J'}$ has the same cluster type as $J$,
namely $A_{n-1}$. 

The classification when $J=I$ (that is, in the case of
$\C[N]$ or $\C[B^-\backslash G]$) was given by Berenstein,
Fomin and Zelevinsky \cite{BFZ}. The only finite type
cases are $A_n\ (n\le 4)$.
The classification when $J=\{j\}$ is a singleton 
and $G$ is of type $A_n$ (the Grassmannian $\Gr(j,n+1)$)
was given by Scott \cite{S}.
When $J=\{1\}$ (the projective space $\Pro(\C^{n+1})$),
the cluster algebra is trivial, since every indecomposable
object of $\Sub Q_1$ is a relative projective.

Note that if $\Sub Q_J$ has finitely many isomorphism classes
of indecomposable objects then by construction ${\cal A}_J$
has finite cluster type. The converse is also true although
not so obvious. Indeed, if ${\cal A}_J$ has finite cluster type,
then by using the classification theorem of Fomin and Zelevinsky
\cite{FZ2} there exists a complete rigid object of $\Sub Q_J$ whose endomorphism
ring has a Gabriel quiver with stable part of Dynkin type.
Using a theorem of Keller and Reiten \cite{KR}, it follows
that the stable category $\underline{\Sub} Q_J$ is triangle
equivalent to a cluster category of Dynkin type, hence 
$\Sub Q_J$ has finitely many indecomposable objects.
Therefore the above classification is also the classification
of all subcategories $\Sub Q_J$ with finitely many indecomposable
objects.

\section{Canonical bases, total positivity and open pro\-blems}

\label{sect18}

Since we started this survey with a discussion of total positivity 
and canonical bases, it is natural to ask if the previous
constructions give a better understanding of these topics.

So let $\BB$ and $\Ss$ denote respectively the dual canonical
and dual semicanonical bases of $\C[N]$. 
We have seen (see \S\ref{sect7}) that for every rigid 
$\Lambda$-module $M$, the function $\varphi_M$ belongs to $\Ss$.

\begin{conjecture}\label{conj3}
For every rigid 
$\Lambda$-module $M$, the function $\varphi_M$ belongs to $\BB$.
\end{conjecture}

The conjecture holds in type $A_n\ (n\le 4)$ \cite{GLSAnnENS}, that
is, when $\Lambda$ has finite representation type. In this case
one even has $\BB = \Ss$.

As explained in \S\ref{sect8}, each finite-dimensional irreducible
$G$-module $L(\lambda)$ has a canonical embedding in $\C[N]$.
It is known that the subsets 
\[
\BB(\lambda) = \BB \cap L(\lambda),
\qquad
\Ss(\lambda) = \Ss \cap L(\lambda),
\]
are bases of $L(\lambda)$. Using the multiplicity-free decomposition (see \S\ref{sect10})
\[
\C[B_K^-\backslash G] = \bigoplus_{\la\in\Pi_J} L(\la),
\]
we therefore obtain a dual canonical and a dual semicanonical basis
of $\C[B_K^-\backslash G]$:
\[
\BB_J=\cup_{\la\in\Pi_J} \BB(\lambda),
\qquad
\Ss_J=\cup_{\la\in\Pi_J} \Ss(\lambda).
\]
It follows from our construction that all the cluster monomials of
the cluster algebra 
$\widetilde{\cal A}_J \subset\C[B_K^-\backslash G]$
belong to $\Ss_J$. 
Conjecture~\ref{conj3} would imply that they also belong to
$\BB_J$.
In particular, when $\widetilde{\cal A}_J$ has finite cluster type,
$\BB_J$ should be equal to the set of cluster monomials.

Regarding total positivity, we propose the following conjecture,
inspired by Fomin and Zelevinsky's approach to total positivity
via cluster algebras.
Let $X$ denote the partial flag variety
$B_K^-\backslash G$ and let $X_{>0}$ be the totally
positive part of $X$ \cite{L3}.
Lusztig has shown that it can be defined by $\dim L(\lambda)$ 
algebraic inequalities given by the elements of $\BB(\lambda)$ 
for a ``sufficiently large''
$\lambda \in \Pi_J$ \cite[Th. 3.4]{L3}.
In fact $X_{>0}\subset N_K$, where $N_K$ is embedded in $X$
as in \S\ref{sect16}. 
We propose the following alternative descriptions of
$X_{>0}$ by systems of $d=\dim X$ algebraic inequalities.
  
\begin{conjecture}\label{conj4}
Let $T=T_1\oplus \cdots \oplus T_d$ be a basic complete rigid  
$\Lambda$-module in $\Sub Q_J$.
Then $x\in N_K$ belongs to $X_{>0}$ if and only if 
\[
\varphi_{T_i}(x) > 0, \qquad (i=1,\ldots,d).
\] 
\end{conjecture}

\begin{example}
{\rm
We consider again type $D_4$ and $J=\{4\}$, so that $X$ can
be identified with the Grassmannian of isotropic lines in
$\C^8$, as in $\S\ref{sect2}$.
In this case $\Pi_J = \N\varpi_4$ and Lusztig's description involves
$\dim L(\lambda)$ inequalities where $\lambda = k\varpi_4$
with $k\ge 5$. For example $\dim L(5\varpi_4)=672$.

On the other hand the category $\Sub Q_4$ has $4$ basic complete
rigid modules:
\[
M_4\oplus M_5\oplus M_6\oplus M_7\oplus M_8\oplus Q_4,\quad 
M_4\oplus M_5\oplus M_6\oplus M_7^*\oplus M_8\oplus Q_4,
\]
\[
M_4\oplus M_5\oplus M_6\oplus M_7\oplus M_8^*\oplus Q_4,\quad 
M_4\oplus M_5\oplus M_6\oplus M_7^*\oplus M_8^*\oplus Q_4,
\]
where we have used the notation of Examples~\ref{ex133}
and \ref{ex15.2}.
Each of them gives rise to a positivity criterion consisting
of $\dim X = 6$ inequalities. 
Using the notation of \S\ref{sect2}, these are respectively
\begin{align*}
y_4&> 0,\ p > 0,\ y_5 > 0,\ y_2 > 0,\ y_3 > 0,\ y_8 > 0;
\\
y_4&> 0,\ p > 0,\ y_5 > 0,\ y_7 > 0,\ y_3 > 0,\ y_8 > 0;
\\
y_4&> 0,\ p > 0,\ y_5 > 0,\ y_2 > 0,\ y_6 > 0,\ y_8 > 0;
\\
y_4&> 0,\ p > 0,\ y_5 > 0,\ y_7 > 0,\ y_6 > 0,\ y_8 > 0.
\end{align*}
Note that since we regard $X_{>0}$ as a subset of $N_K$,
the additional relation $y_1=1$ is understood.
Thus the conjecture holds in this case, and more generally
in type $D_n$ when $J=\{n\}$.
}
\end{example}
When $X=B^-\backslash G$ and $M$ is a rigid module in $\R$, the conjecture
follows from the work of Berenstein, Fomin and Zelevinsky and
our construction.
When $X$ is a type $A$ Grassmannian and $M$ is a rigid module in $\R$, 
the conjecture follows from the work of Scott \cite{S} and 
our construction.


\end{document}